\documentclass{amsart}
\usepackage{amsfonts}
\usepackage{amsmath,amscd}
\usepackage{amssymb}
\usepackage{amsthm}
\usepackage{newlfont}
\allowdisplaybreaks

\long\def\alert#1{\parindent2em\smallskip\hbox to\hsize
{\hskip\parindent\vrule
\vbox{\advance\hsize-2\parindent\hrule\smallskip\parindent.4\parindent
\narrower\noindent#1\smallskip\hrule}\vrule\hfill}\smallskip\parindent0pt}
 \newtheorem{thm}{Theorem}[section]
\newtheorem{cor}[thm]{Corollary}
 \newtheorem{lem}[thm]{Lemma}
 
\theoremstyle{definition}
 \newtheorem{defn}[thm]{Definition}
\theoremstyle{remark}

 \numberwithin{equation}{section}

\begin{document}


%
\title[ On the non-abelian tensor square of groups of order dividing $p^{5}$] {On the non-abelian tensor square of all groups of order dividing $p^{5}$}
\author[T. J. Ghorbanzadeh]{Taleea jalaeeyan ghorbanzadeh}
\author[M. Parvizi]{Mohsen Parvizi}
\author[P. Niroomand]{Peyman Niroomand}

\address{Department of Pure Mathematics\\
Ferdowsi University of Mashhad, Mashhad, Iran}
\email{jalaeeyan@gmail.com, ta.jalaeeyan@mail.um.ac.ir}
\address{Department of Pure Mathematics\\
Ferdowsi University of Mashhad, Mashhad, Iran}
\email{parvizi@math.um.ac.ir}
\address{School of Mathematics and Computer Science\\
Damghan University, Damghan, Iran}
\email{niroomand@du.ac.ir, p$\_$niroomand@yahoo.com}

%



\begin{abstract}
In this paper we consider all groups of order dividing $p^5$.  We obtain  the explicit structure of the non-abelian tensor square, non-abelian exterior square, tensor center, exterior center, the third homotopy group of suspension of an Eilenberg-MacLain space $k(G,1) $ and $\triangledown(G)$ of such groups.
\end{abstract}
\maketitle
\section{Introduction and motivation}
The concept of non-abelian tensor product of groups was born by Brown and Loday \cite{Brown1, B2}, which is an applied topic in $K$-theory and homotopy theory. The readers can find some spacious result on this topic in  \cite{ Blyth, Brown1, Brown} and \cite{Moghaddam}. The non-abelian tensor square $G\otimes G$ of the group $G$ is the group generated by the symbols $g\otimes h$, subject to the relations
\[gg' \otimes h =\  (^{g}g' \otimes ^{g} h)(g \otimes h) ~\text{and}~ g \otimes hh' = (g \otimes h) ( ^h g \otimes ^h h')\]
for all $g, g',h,h'\in G$, where $G$ acts  on itself by conjugation via $^{g}g'=gg'g^{-1}$.\newline The tensor square is a special case of non-abelian tensor product $G \otimes H$ of two arbitrary groups $G$ and $H$. The exterior square $ G \wedge G$ is obtained by imposing the additional relation $g\otimes g =1_{\otimes} $ on $G \otimes G.$\newline
Recall that \cite{Brown} describes the maps $\kappa : G \otimes G\rightarrow G'$ and  $\kappa' : G \wedge G\rightarrow G'$ which are both homomorphisms of groups. The kernel of $\kappa$ which is denoted by $J_{2}(G)$, is isomorphic to the third homotopy group of suspension of an Eilenberg-MacLain space $k(G,1)$ and the kernel of $\kappa'$ which is isomorphic to $\mathcal{M}(G)$, is the Schur multiplier of $G$ (see \cite{B2, Brown} for more details).
The following commutative diagram with exact rows and central extensions as columns involves the third integral homology group of $G$ and  Whitehead functor (see \cite{Brown} for details)
\begin{align*}
&\quad  \quad \quad\quad \ \quad  \quad  \quad  \quad \quad\ \quad  \quad  \quad  \quad \quad\ 0  \quad  \quad  \quad \quad 0\\
&\quad  \quad \quad\quad \ \quad  \quad  \quad  \quad \quad\ \quad  \quad  \quad  \quad \quad\downarrow  \quad  \quad  \quad \quad\downarrow\\
&1\longrightarrow  H_{3} (G)\longrightarrow\Gamma(G/G')\longrightarrow J_{2}(G)\longrightarrow H_{2} (G)\longrightarrow 0 \\&\quad  \quad \quad\quad \parallel \quad  \quad  \quad  \quad \quad\downarrow \quad  \quad  \quad  \quad \quad\downarrow  \quad  \quad  \quad \quad\downarrow\\
& \quad \quad  \quad H_{3} (G)\longrightarrow \Gamma(G/G')\longrightarrow G\otimes G\longrightarrow G\wedge G\longrightarrow 1 \\
&\quad  \quad \quad\quad \quad  \quad  \quad  \quad \quad \quad  \quad  \quad  \quad \quad \kappa\downarrow   \quad \quad  \quad \kappa'\downarrow\\
&\quad  \quad \quad\quad \ \quad  \quad  \quad \quad \quad\ \quad  \quad  \quad  \quad \quad\ G'   \quad =\quad   G'\\
&\quad  \quad \quad\quad \ \quad  \quad  \quad  \quad \quad\ \quad  \quad  \quad  \quad \quad \downarrow \quad \quad\quad \quad  \downarrow\\
&\quad  \quad \quad\quad \ \quad  \quad  \quad  \quad \quad\ \quad  \quad  \quad  \quad \quad\ 0  \quad  \quad  \quad \quad 0\\
\end{align*}
Following the terminology in \cite{Ellis5}, we consider the notations of  tensor center and exterior center,
 $Z^{\otimes}(G)=\lbrace g\in G\vert g\otimes h=1_{G\otimes G}$, for all $ h \in G\rbrace,$ and
$Z^{\wedge}(G) =\lbrace g\in G\vert g\wedge h=1_{G\wedge G},$ for all $ h \in G\rbrace$, respectively. $Z^{\wedge}(G)$ is a subgroup of $Z(G)$ which has the property that $G$ is capable if and only if $Z^{\wedge}(G)= 1$ that is, whether $G$ is isomorphic to $E/Z(E)$ for some group $E$.
 Ellis proved $Z^{\wedge}(G)$ is isomorphic to the epicenter of $G$ which is denoted by $Z^{*}(G)$, that is defined as follows
 \begin{defn}
$Z^{\ast}(G)$ is the intersection of all subgroups of the form $\psi(Z(G))$ where $\psi : E \rightarrow G$ is an arbitrary surjective homomorphism with ker $\psi\subseteq Z(G)$.
 \end{defn}
The next lemmas show the relations between  $Z^{*}(G)$, $Z^{\otimes}(G)$ and $Z^{\wedge}(G)$, which play important role in this paper
\begin{lem} \cite[Proposition 16 $(vii)$]{Ellis5}\label{2.1} Let $G$ be any group. Then $Z^{*}(G) \cong Z^{\wedge}(G).$
\end{lem}
\begin{lem}\cite[Propositin 16 $(ii)$]{Ellis5}\label{88}
Let $G$ be any group. Then
$Z^{\otimes}(G) \leq Z^{\wedge}(G).$
\end{lem}
 Thanks to \cite{Girant}, by the next theorem we give the presentation of all groups of order dividing $p^5$.


\begin{thm}\cite{Girant} \label{main}
Let $w$ be a primitive root of finite domain $\mathbb{F}_{p}$ of order $p $  $(p>3)$ and let $a,b$ and $k$ be integers where $a\in w_{3}=\lbrace x\in \mathbb{F}_{p}\vert x^{3}=1\rbrace, b\in w_{4}=\lbrace x\in \mathbb{F}_{p}\vert x^{4}=1\rbrace$ and $k\in \lbrace 1\dots ,\frac{p-1}{2}\rbrace$.  Let $G_{i}= \langle g_{1}, g_{2}, g_{3}, g_{4}, g_{5}\vert R\rangle$, since this groups are polycyclic, they have polycyclic presentation, so they satisfy:

$g_{i}^{r_{i}}=g_{i+1}^{a_{i+1}} \dots g_{5}^{a_{5}}$

$\left[g_{i}, g_{j}\right]=g_{j+1}^{a_{i+j, j+1}} \dots g_{5}^{a_{i+j, 5}}$.\newline
For the summery we omitted the relations:
$[g_{i}, g_{j}]=1$ and ${g_{i}}^{p}=1$ for any $i,j$ where $1\leqslant i, j \leqslant 5$. Then:

\begin{align*}
&G_{1}= \langle g_{1}, g_{2}, g_{3}, g_{4}, g_{5}\vert g_{1}^{p}=g_{2}, g_{2}^{p}=g_{3}, g_{3}^p=g_{4}, g_{4}^{p}=g_{5}\rangle \\
&G_{2}= \langle g_{1}, g_{2} , g_{3},g_{4},g_{5}\vert  [g_{2},g_{1}]=g_{3}, g_{1}^{p}=g_{4}, g_{2}^{p}=g_{5}\rangle\\
&G_{3}=   \langle g_{1}, g_{2}, g_{3}, g_{4}, g_{5}\vert  [g_{2},g_{1}]=g_{3},[g_{3},g_{1}]=g_{4}, [g_{3},g_{2}]=g_{5}\rangle\\
&G_{4} =  \langle g_{1}, g_{2}, g_{3}, g_{4}, g_{5}\vert[g_{2},g_{1}]=g_{3}, [g_{3},g_{1}]=g_{4}, [g_{3},g_{2}]=g_{5},  g_{2}^{p}=g_{5}\rangle \\
 &G_{5}= \langle g_{1}, g_{2}, g_{3}, g_{4}, g_{5}\vert[g_{2},g_{1}]=g_{3}, [g_{3},g_{1}]=g_{4}, [g_{3},g_{2}]=g_{5},  g_{2}^{p}=g_{4}\rangle \\
 &G_{6}= \langle g_{1}, g_{2}, g_{3}, g_{4}, g_{5}\vert[g_{2},g_{1}]=g_{3}, [g_{3},g_{1}]=g_{4}, [g_{3},g_{2}]=g_{5},  g_{2}^{p}=g_{4}^{w}\rangle \\
  &G_{7} = \langle g_{1}, g_{2}, g_{3}, g_{4}, g_{5}\vert[g_{2},g_{1}]=g_{3}, [g_{3},g_{1}]=g_{4}, [g_{3},g_{2}]=g_{5},  g_{1}^{p}=g_{4},  g_{2}^{p}=g_{5}\rangle\\
 &G_{8}=\langle g_{1}, g_{2}, g_{3}, g_{4}, g_{5}\vert[g_{2},g_{1}]=g_{3}, [g_{3},g_{1}]=g_{4}g_{5}, [g_{3},g_{2}]=g_{5},  g_{1}^{p}=g_{4},  g_{2}^{p}=g_{5}\rangle\\ &G_{9}=\langle g_{1}, g_{2}, g_{3}, g_{4}, g_{5}\vert[g_{2},g_{1}]=g_{3}, [g_{3},g_{1}]=g_{4}g_{5}^{w}, [g_{3},g_{2}]=g_{5},  g_{1}^{p}=g_{4},  g_{2}^{p}=g_{5}\rangle\\
 &G_{10}= \langle g_{1}, g_{2}, g_{3}, g_{4}, g_{5}\vert[g_{2},g_{1}]=g_{3}, [g_{3},g_{1}]=g_{5}^{w}, [g_{3},g_{2}]=g_{4},  g_{1}^{p}=g_{4},  g_{2}^{p}=g_{5}\rangle\\  &G_{11_{k}} = \langle g_{1}, g_{2}, g_{3}, g_{4}, g_{5}\vert[g_{2},g_{1}]=g_{3}, [g_{3},g_{1}]=g_{4}, [g_{3},g_{2}]=g_{5}^{w^{k}},  g_{1}^{p}=g_{4},  g_{2}^{p}=g_{5}\rangle\\
&G_{12_{k}} = \langle g_{1}, g_{2}, g_{3}, g_{4}, g_{5}\vert[g_{2},g_{1}]=g_{3}, [g_{3},g_{1}]=g_{4}g_{5}^{w^{k}}, [g_{3},g_{2}]=g_{4}^{w^{k-1}}g_{5},  g_{1}^{p}=g_{4},  g_{2}^{p}=g_{5}\rangle\\
&G_{13} = \langle g_{1}, g_{2}, g_{3}, g_{4}, g_{5}\vert g_{1}^{p}=g_{3}, g_{2}^{p}=g_{4}, g_{3}^{p}=g_{5}\rangle\cong C_{p^{3}}\times C_{p^{2}}\\
 &G_{14} = \langle g_{1}, g_{2}, g_{3}, g_{4}, g_{5}\vert [g_{2},g_{1}]=g_{5}, g_{1}^{p}=g_{3}, g_{2}^{p}=g_{4}, g_{3}^{p}=g_{5}\rangle\\
&G_{15} =  \langle g_{1}, g_{2}, g_{3}, g_{4}, g_{5}\vert [g_{2},g_{1}]=g_{3}, g_{1}^{p}=g_{4}, g_{4}^{p}=g_{5}\rangle\\
& G_{16} = \langle g_{1}, g_{2}, g_{3}, g_{4}, g_{5}\vert [g_{2},g_{1}]=g_{3}, g_{1}^{p}=g_{4}, [g_{3},g_{1}]=g_{5} \rangle\\
&G_{17} = \langle g_{1}, g_{2}, g_{3}, g_{4}, g_{5}\vert [g_{2},g_{1}]=g_{3}, g_{1}^{p}=g_{4}, [g_{3},g_{1}]=g_{5},g_{2}^{p}=g_{5} \rangle\\
 &G_{18}= \langle g_{1}, g_{2}, g_{3}, g_{4}, g_{5}\vert [g_{2},g_{1}]=g_{3}, g_{1}^{p}=g_{4}, [g_{3},g_{1}]=g_{5}^{w},g_{2}^{p}=g_{5} \rangle\\
&G_{19} = \langle g_{1}, g_{2}, g_{3}, g_{4}, g_{5}\vert [g_{2},g_{1}]=g_{3}, g_{1}^{p}=g_{4}, [g_{3},g_{1}]=g_{5}, g_{4}^{p}=g_{5} \rangle\\
 &G_{20}=  \langle g_{1}, g_{2}, g_{3}, g_{4}, g_{5}\vert [g_{2},g_{1}]=g_{3}, g_{1}^{p}=g_{4}, [g_{3},g_{1}]=g_{5}, [g_{3},g_{2}]=g_{5} \rangle\\
 &G_{21}= \langle g_{1}, g_{2}, g_{3}, g_{4}, g_{5}\vert [g_{2},g_{1}]=g_{3}, g_{1}^{p}=g_{4}, [g_{3},g_{1}]=g_{5}, [g_{3},g_{2}]=g_{5},  g_{2}^{p}=g_{5}\rangle\\
&G_{22}= \langle g_{1}, g_{2}, g_{3}, g_{4}, g_{5}\vert [g_{2},g_{1}]=g_{3}, g_{1}^{p}=g_{4}, [g_{3},g_{1}]=g_{5}, [g_{3},g_{2}]=g_{5},  g_{4}^{p}=g_{5}\rangle\\
 &G_{23}= \langle g_{1}, g_{2}, g_{3}, g_{4}, g_{5}\vert [g_{2},g_{1}]=g_{3}, g_{1}^{p}=g_{4}, [g_{3},g_{1}]=g_{5}^{w}, [g_{3},g_{2}]=g_{5}^{w},  g_{4}^{p}=g_{5}\rangle\\
&G_{24}=  \langle g_{1},  g_{2}, g_{3}, g_{4}, g_{5}\vert [g_{2},g_{1}]=g_{3},g_{1}^{p}=g_{4}, g_{2}^{p}=g_{3}, [g_{3},g_{1}]=g_{5}, [g_{4},g_{2}]=g_{5}^{p-1},  g_{3}^{p}=g_{5}\rangle\\
&G_{25} = \langle g_{1},  g_{2}, g_{3}, g_{4}, g_{5}\vert [g_{2},g_{1}]=g_{3},g_{1}^{p}=g_{4}, g_{2}^{p}=g_{3}, g_{4}^{p}=g_{5}\rangle\\
 &G_{26}=  \langle g_{1},  g_{2}, g_{3}, g_{4}, g_{5}\vert g_{1}^{p}=g_{3}, g_{3}^{p}=g_{4}, g_{4}^{p}=g_{5}\rangle\cong C_{p^{4}}\times C_{p}\\
&G_{27}=  \langle g_{1},  g_{2}, g_{3}, g_{4}, g_{5}\vert g_{1}^{p}=g_{3}, g_{3}^{p}=g_{4}, g_{4}^{p}=g_{5}, [g_{2},g_{1}]=g_{5}\rangle\\
&G_{28} =  \langle g_{1},  g_{2}, g_{3}, g_{4}, g_{5}\vert [g_{2},g_{1}]=g_{3}, [g_{3},g_{1}]=g_{4}, [g_{4},g_{1}]=g_{5}\rangle\\
 &G_{29_{a}}= \langle g_{1},  g_{2}, g_{3}, g_{4}, g_{5}\vert [g_{2},g_{1}]=g_{3}, [g_{3},g_{1}]=g_{4}, [g_{4},g_{1}]=g_{5}^{w^{a}}, g_{2}^{p}=g_{5}\rangle\\
&G_{30} =\langle g_{1},  g_{2}, g_{3}, g_{4}, g_{5}\vert [g_{2},g_{1}]=g_{3}, [g_{3},g_{1}]=g_{4}, [g_{4},g_{1}]=g_{5}, g_{1}^{p}=g_{5}\rangle\\
&G_{31}= \langle g_{1},  g_{2}, g_{3}, g_{4}, g_{5}\vert [g_{2},g_{1}]=g_{3}, [g_{3},g_{1}]=g_{4}, [g_{4},g_{1}]=g_{5},  [g_{3},g_{2}]=g_{5}\rangle\\
&G_{32_{a}}= \langle g_{1},  g_{2}, g_{3}, g_{4}, g_{5}\vert [g_{2},g_{1}]=g_{3}, [g_{3},g_{1}]=g_{4}, [g_{4},g_{1}]=g_{5}^{w^{a}},  [g_{3},g_{2}]=g_{5}^{w^{a}},g_{2}^{p}=g_{5}\rangle\\
& G_{33_{b}}=\langle g_{1},  g_{2}, g_{3}, g_{4}, g_{5}\vert [g_{2},g_{1}]=g_{3}, [g_{3},g_{1}]=g_{4}, [g_{4},g_{1}]=g_{5}^{w^{b}},  [g_{3},g_{2}]=g_{5}^{w^{b}},g_{1}^{p}=g_{5}\rangle\\
&G_{34} = \langle g_{1},  g_{2}, g_{3}, g_{4}, g_{5}\vert [g_{2},g_{1}]=g_{4}, [g_{3},g_{1}]=g_{5}\rangle\\
&G_{35}= \langle g_{1},  g_{2}, g_{3}, g_{4}, g_{5}\vert [g_{2},g_{1}]=g_{4}, g_{3}^{p}=g_{5}\rangle\\
 &G_{36}= \langle g_{1},  g_{2}, g_{3}, g_{4}, g_{5}\vert [g_{2},g_{1}]=g_{4}, [g_{3},g_{2}]=g_{5}, g_{3}^{p}=g_{5}\rangle\\
&G_{37} =\langle g_{1},  g_{2}, g_{3}, g_{4}, g_{5}\vert [g_{3},g_{2}]=g_{4}, g_{3}^{p}=g_{5}\rangle\\
&G_{38} =\langle g_{1},  g_{2}, g_{3}, g_{4}, g_{5}\vert [g_{3},g_{1}]=g_{4}, [g_{3},g_{2}]=g_{5}, g_{3}^{p}=g_{5}\rangle\\
&G_{39} =\langle g_{1},  g_{2}, g_{3}, g_{4}, g_{5}\vert [g_{2},g_{1}]=g_{4}, [g_{3},g_{2}]=g_{5}, g_{3}^{p}=g_{4}\rangle\\
&G_{40}=\langle g_{1},  g_{2}, g_{3}, g_{4}, g_{5}\vert [g_{3},g_{2}]=g_{5}, g_{2}^{p}=g_{4}, g_{3}^{p}=g_{5}\rangle\\
&G_{41}= \langle g_{1},  g_{2}, g_{3}, g_{4}, g_{5}\vert [g_{3},g_{1}]=g_{4}, [g_{3},g_{2}]=g_{5}, g_{2}^{p}=g_{4}, g_{3}^{p}=g_{5}\rangle\\
&G_{42}=\langle g_{1},  g_{2}, g_{3}, g_{4}, g_{5}\vert [g_{2},g_{1}]=g_{4}, [g_{3},g_{2}]=g_{5}, g_{2}^{p}=g_{4}, g_{3}^{p}=g_{5}\rangle\\
&G_{43}=\langle g_{1},  g_{2}, g_{3}, g_{4}, g_{5}\vert  g_{2}^{p}=g_{4}, g_{3}^{p}=g_{5}\rangle\\
&G_{44}=\langle g_{1},  g_{2}, g_{3}, g_{4}, g_{5}\vert [g_{2},g_{1}]=g_{4}, [g_{3},g_{1}]=g_{5}, g_{2}^{p}=g_{4}, g_{3}^{p}=g_{5}\rangle\\
 &G_{45} =\langle g_{1},  g_{2}, g_{3}, g_{4}, g_{5}\vert [g_{2},g_{1}]=g_{5},  g_{2}^{p}=g_{4}, g_{3}^{p}=g_{5}\rangle\\
&G_{46} =\langle g_{1},  g_{2}, g_{3}, g_{4}, g_{5}\vert [g_{2},g_{1}]=g_{4}g_{5},[g_{3},g_{1}]=g_{5} , g_{2}^{p}=g_{4}, g_{3}^{p}=g_{5}\rangle\\
& G_{47}=\langle g_{1},  g_{2}, g_{3}, g_{4}, g_{5}\vert [g_{3},g_{1}]=g_{5}, g_{2}^{p}=g_{4}, g_{3}^{p}=g_{5}\rangle\\
&G_{48_{k}} =\langle g_{1},  g_{2}, g_{3}, g_{4}, g_{5}\vert [g_{2},g_{1}]=g_{4}, [g_{3},g_{1}]=g_{5}^{w^{k}},   g_{2}^{p}=g_{4}, g_{3}^{p}=g_{5}\rangle\\
 &G_{49}= \langle g_{1},  g_{2}, g_{3}, g_{4}, g_{5}\vert [g_{2},g_{1}]=g_{5}^{w^{k}}, [g_{3},g_{1}]=g_{4},   g_{2}^{p}=g_{4}, g_{3}^{p}=g_{5}\rangle\\
&G_{50_{k}}=\langle g_{1},  g_{2}, g_{3}, g_{4}, g_{5}\vert [g_{2},g_{1}]=g_{4}g_{5}^{w^{k}}, [g_{3},g_{1}]=g_{4}^{w^{k-1}}g_{5},   g_{2}^{p}=g_{4}, g_{3}^{p}=g_{5}\rangle\\
&G_{51}=\langle g_{1},  g_{2}, g_{3}, g_{4}, g_{5}\vert   g_{1}^{p}=g_{4}, g_{4}^{p}=g_{5}\rangle\\
 &G_{52}=\langle g_{1},  g_{2}, g_{3}, g_{4}, g_{5}\vert   g_{1}^{p}=g_{4}, g_{4}^{p}=g_{5}, [g_{3},g_{2}]=g_{5}\rangle\\
 &G_{53}=\langle g_{1},  g_{2}, g_{3}, g_{4}, g_{5}\vert   g_{1}^{p}=g_{4}, g_{4}^{p}=g_{5}, [g_{3},g_{1}]=g_{5}\rangle\\
&G_{54} =\langle g_{1},  g_{2}, g_{3}, g_{4}, g_{5}\vert  [g_{2},g_{1}]=g_{4}, [g_{4},g_{2}]=g_{5}\rangle\\
&G_{55}=\langle g_{1},  g_{2}, g_{3}, g_{4}, g_{5}\vert  [g_{2},g_{1}]=g_{4}, [g_{4},g_{2}]=g_{5},  g_{3}^{p}=g_{5}\rangle\\
&G_{56} =\langle g_{1}, g_{2}, g_{3}, g_{4}, g_{5}\vert  [g_{2},g_{1}]=g_{4}, [g_{4},g_{2}]=g_{5},  g_{2}^{p}=g_{5}\rangle\\
 &G_{57}= \langle g_{1}, g_{2}, g_{3}, g_{4}, g_{5}\vert  [g_{2}, g_{1}]=g_{4}, [g_{4},g_{2}]=g_{5},  g_{1}^{p}=g_{5}\rangle\\
&G_{58}=\langle g_{1}, g_{2}, g_{3}, g_{4}, g_{5}\vert  [g_{2}, g_{1}]=g_{4}, [g_{4},g_{2}]=g_{5}^{w},  g_{1}^{p}=g_{5}\rangle\\
&G_{59}= \langle g_{1}, g_{2}, g_{3}, g_{4}, g_{5}\vert  [g_{2}, g_{1}]=g_{4}, [g_{4},g_{2}]=g_{5}, [g_{3}, g_{1}]=g_{5}\rangle\\
&G_{60}=\langle g_{1}, g_{2}, g_{3}, g_{4}, g_{5}\vert  [g_{2}, g_{1}]=g_{4}, [g_{4},g_{2}]=g_{5}, [g_{3},g_{1}]=g_{5}, g_{3}^{p}=g_{5}\rangle\\
&G_{61} =\langle g_{1},  g_{2}, g_{3}, g_{4}, g_{5}\vert  [g_{2}, g_{1}]=g_{4}, [g_{4}, g_{2}]=g_{5}, [g_{3},g_{1}]=g_{5}, g_{2}^{p}=g_{5}\rangle\\
&G_{62}= \langle g_{1},  g_{2}, g_{3}, g_{4}, g_{5}\vert  [g_{2}, g_{1}]=g_{4}, [g_{4}, g_{2}]=g_{5}, [g_{3},g_{1}]=g_{5}, g_{1}^{p}=g_{5}\rangle\\
&G_{63}=\langle g_{1},  g_{2}, g_{3}, g_{4}, g_{5}\vert  [g_{2}, g_{1}]=g_{4}, [g_{4}, g_{2}]=g_{5}^{w}, [g_{3},g_{1}]=g_{5}^{w}, g_{1}^{p}=g_{5}\rangle\\
&G_{64}=\langle g_{1},  g_{2}, g_{3}, g_{4}, g_{5}\vert  [g_{2}, g_{1}]=g_{5}\rangle\\
 &G_{65}= \langle g_{1},  g_{2}, g_{3}, g_{4}, g_{5}\vert  [g_{2}, g_{1}]=g_{5},  [g_{4}, g_{3}]=g_{5}\rangle\\
&G_{66}= \langle g_{1},  g_{2}, g_{3}, g_{4}, g_{5}\vert  g_{1}^{p}=g_{5} \rangle\\
 &G_{67}= \langle g_{1},  g_{2}, g_{3}, g_{4}, g_{5}\vert   [g_{2}, g_{1}]=g_{5}, g_{1}^{p}=g_{5} \rangle\\
&G_{68}= \langle g_{1},  g_{2}, g_{3}, g_{4}, g_{5}\vert   [g_{2}, g_{1}]=g_{5}, g_{3}^{p}=g_{5} \rangle\\
&G_{69}= \langle g_{1},  g_{2}, g_{3}, g_{4}, g_{5}\vert   [g_{2}, g_{1}]=g_{5},  [g_{4}, g_{3}]=g_{5}, g_{4}^{p}=g_{5} \rangle\\
 &G_{70}\cong C_{p}^5
\end{align*}
\end{thm}
This paper is devoted  to obtain the structure of $G\wedge G, Z^{\wedge}(G), G\otimes G, Z^{\otimes}(G),$\newline $\Pi_{3}(SK(G,1))$ the third homotopy group of suspension of an Eilenberg-MacLain space $k(G,1)$ and  $\triangledown(G)$ for all groups of order dividing $p^5$.\newline The authors in \cite{Ghorbanzadeh} obtained the structure of $ G\wedge G, Z^{\wedge}(G), G\otimes G, Z^{\otimes}(G)$ and $\Pi_{3}(SK(G,1))$ for all groups of order $p^4$.\newline The same motivation allows us to write the present article. It is instructive to note that  at the same time and in process of preparation and finalizing  this article we noticed \cite{Hatui}. \newline But it is worth noting however Hatui et al. have the same results originated from the classification of James in \cite{James}, they claimed the presentation of their article  is a bit unusual while we obtained our results straightforward.\newline Our work is  more shorter and the readers can find it more understandable while we get the results with different argument. Just to obtain the non-abelian exterior square of groups $G_{2}, G_{10}, G_{11,2}, G_{17}, G_{18}, G_{28}, G_{31}, G_{40}, G_{41}, G_{45},  G_{48,2}$  and $ G_{49}$ due to same argument, we prefer to refer to \cite{Hatui}.
\newline
The following theorems from \cite{Horn} give us the structure of the Schur multiplier and the epicenter of all groups of order dividing $p^5$.
\begin{thm} \label{Schur}
 Let $G$ be a group of order dividing $p^5$ where $p>3$ is an odd prime. Then
  \[\mathcal{M}(G)\cong\left \{ \begin{array}{ll}
     1 & \textrm{if $G \cong$ }~\textrm{$G_{1},$}~\textrm{$G_{7},\dots ,G_{9}, $}~\textrm{$G_{11,1},$}~\textrm{$G_{12_{k}},$}~\textrm{$G_{24},$}~\textrm{$G_{27}.$}\\
      \mathbb{Z}_{p} & \textrm{if $G \cong$ }~\textrm{$G_{4},\dots ,G_{6}, $}~\textrm{$G_{10},$}~\textrm{$G_{11,2},$}~\textrm{$G_{12_{k}},$}~\textrm{$G_{19},$}~\textrm{$G_{21},$}~\textrm{$G_{22},$}~\textrm{$G_{23},$}\\ &~\textrm{$G_{25},$}~\textrm{$G_{29_{a}}, $}~\textrm{$G_{30}, $}~\textrm{$G_{32_{a}}, $}~\textrm{$G_{33_{b}}, $}~\textrm{$G_{42}, $}~\textrm{$G_{44}, $}~\textrm{$G_{46}, $}~\textrm{$G_{48,1}, $}~\textrm{$G_{50_{k}}. $}
      \\ \mathbb{Z}_{p}^{(2)} & \textrm{if $G \cong$ }~\textrm{$G_{15},\dots ,G_{18}, $}~\textrm{$G_{20},$}~\textrm{$G_{41},$}~\textrm{$G_{47}, $}~\textrm{$G_{52}, $}~\textrm{$G_{53}. $}\\
       \mathbb{Z}_{p}^{(3)}  & \textrm{if $G \cong$ }~\textrm{$G_{2},$}~\textrm{$G_{3},$}~\textrm{$G_{28},$}~\textrm{$G_{31},$}~\textrm{$G_{36},$}~\textrm{$G_{38},$}~\textrm{$G_{39},$}~\textrm{$G_{40},$}~\textrm{$G_{51},$}\\ &~\textrm{$G_{55},\dots ,G_{58}, $}~\textrm{$G_{60},\dots ,G_{63}. $}\\
   \mathbb{Z}_{p}^{(4)} &\textrm{if $G \cong$ }~\textrm{$G_{35},$}~\textrm{$G_{37},$}~\textrm{$G_{54},$}~\textrm{$G_{59}.$}\\
   \mathbb{Z}_{p}^{(5)}  & \textrm{if $G \cong$ }~\textrm{$G_{65},$}~\textrm{$G_{67},\dots ,G_{69}. $}\\
    \mathbb{Z}_{p}^{(6)}  & \textrm{if $G \cong$ }~\textrm{$G_{34},$}~\textrm{$G_{66}. $}\\
     \mathbb{Z}_{p}^{(7)}  & \textrm{if $G \cong$ }~\textrm{$G_{64}.$}\\
      \mathbb{Z}_{p}^{(10)}  & \textrm{if $G \cong$}~\textrm{$G_{70}.$}\\
      \mathbb{Z}_{p^2}  & \textrm{if $G \cong$ }~\textrm{$G_{13},$}~\textrm{$G_{48,2}, $}~\textrm{$G_{49}.$}\\
       \mathbb{Z}_{p^2}\oplus \mathbb{Z}_{p} & \textrm{if $G \cong$ }~\textrm{$G_{45}.$}\\
       \mathbb{Z}_{p^2}\oplus \mathbb{Z}_{p}^{(2)}  & \textrm{if $G \cong$ }~\textrm{$G_{43}.$}\\
\end{array} \right.  \]
\end{thm}

\begin{thm}
 \label{1.6}
 Let $G$ be a group of order dividing $p^5$ where $p>3$ is an odd prime. Then
  \[Z^{*}(G)\cong\left \{ \begin{array}{ll}
     1 & \textrm{if $G \cong$ }~\textrm{$G_{2},$}~\textrm{$G_{3}$}~\textrm{$G_{10},$}~\textrm{$G_{11,2},$}~\textrm{$G_{17},$}~\textrm{$G_{18},$}~\textrm{$G_{28},$}~\textrm{$G_{31},$}~\textrm{$G_{34},$}~\textrm{$G_{40},$}~\textrm{$G_{41},$}\\ &~\textrm{$G_{43},$}~\textrm{$G_{45},$}~\textrm{$G_{48,2},$}~\textrm{$G_{49},$}~\textrm{$G_{54},$}~\textrm{$G_{59},$}~\textrm{$G_{17},$}~\textrm{$G_{64},$}~\textrm{$G_{70}.$}\\
      \mathbb{Z}_{p} & \textrm{if $G \cong$ }~\textrm{$G_{4},\dots ,G_{6}, $}~\textrm{$G_{10},$}~\textrm{$G_{14},$}~\textrm{$G_{16},$}~\textrm{$G_{25},$}~\textrm{$G_{29_{a}},$}~\textrm{$G_{30},$}\\ &~\textrm{$G_{32_{a}},$}~\textrm{$G_{33_{b}},$}~\textrm{$G_{35},$}~\textrm{$G_{36},\dots ,G_{39}, $}~\textrm{$G_{55},\dots ,G_{58}, $}~\textrm{$G_{61},\dots ,G_{63}, $}~\textrm{$G_{65},\dots ,G_{69}. $}
      \\ \mathbb{Z}_{p}^{(2)} & \textrm{if $G \cong$ }~\textrm{$G_{7},\dots ,G_{9}, $}~\textrm{$G_{11,1},$}~\textrm{$G_{12_{k}},$}~\textrm{$G_{21}, $}~\textrm{$G_{42}, $}~\textrm{$G_{44}, $}~\textrm{$G_{46}, $}~\textrm{$G_{47}, $}~\textrm{$G_{48,1}, $}\\ &~\textrm{$G_{50_{k}}.$}\\
     \mathbb{Z}_{p^2}  & \textrm{if $G \cong$ }~\textrm{$G_{15},$}~\textrm{$G_{20}, $}~\textrm{$G_{22}, $}~\textrm{$G_{23}, $}~\textrm{$G_{51},\dots ,G_{53}. $}\\
       \mathbb{Z}_{p^3} & \textrm{if $G \cong$ }~\textrm{$G_{26}, $}~\textrm{$G_{27}.$}\\
       \mathbb{Z}_{p^5}  & \textrm{if $G \cong$ }~\textrm{$G_{70}.$}\\
\end{array} \right.  \]
 \end{thm}
\section{$\triangledown(G)$ of all groups of order dividing $p^5$}
In this section, we are going to obtain the structure of $\triangledown(G)$ of groups of order dividing $p^5$ by the following theorems

\begin{thm}\cite[Theorem 1.3. (iii)]{Blyth} \label{2.1}
If $G^{ab}$ has no elements of order $2$, then $\triangledown(G)\cong \Gamma (G^{ab}).$
\end {thm}
Given an abelian group $ A,$ from \cite{Whitehead}, $\Gamma(A)$ is used to denote the abelian group with generators $\gamma(a)$, for $a\in A$, by defining relations \begin{itemize}

\item[$(i).$]$\gamma(a^{-1})=\gamma (a).$
\item[$(ii).$] $\gamma(abc)  \gamma (a) \gamma (b)  \gamma (c) =  \gamma (ab)\gamma (bc) \gamma (ca),$
  \end{itemize}
for all $a, b, c \in A.$ $\Gamma$ is called Whitehead's universal quadratic functor.\newline From \cite{Brown}, we have
\begin{thm}\label{even}
 Let $G$ and $H$ be abelian groups. Then
\begin{itemize}
\item[(i)] $\Gamma (G \times H) \cong \Gamma (G)\times \Gamma (H)\times  (G\otimes H)$,
\item[(ii)] \[ \Gamma (\mathbb{Z}_{n})= \left\{ \begin{array}{ll}
\mathbb{Z}_{n} & \textrm{$n$~is~odd }\\
\mathbb{Z}_{2n} & \textrm{$n$~is~even }\\
\end{array} \right. \]
\end{itemize}
where $\mathbb{Z}_{n} = \langle x\vert x^{n}=e\rangle$ for $n\geq 0$. 
\end{thm}
 First we give the structure of $\triangledown(G)$ when $G$ is an abelian  group of order dividing $p^5.$
\begin{lem}\label{abnabla}
 Let $G$ be an abelian  group of order dividing $p^5$, where $p>3$ is an odd prime. Then
  \[\triangledown(G)\cong\left \{ \begin{array}{ll}
     \mathbb{Z}_{p^5} & \textrm{if $G \cong$ }~\textrm{$G_{1}.$}\\
      \mathbb{Z}_{p^3}\oplus\mathbb{Z}_{p^2}^{(2)} & \textrm{if $G \cong$ }~\textrm{$G_{13}.$}\\ \mathbb{Z}_{p^4}\oplus\mathbb{Z}_{p}^{(2)} & \textrm{if $G \cong$ }~\textrm{$G_{26}.$}\\
    \mathbb{Z}_{p^2}^{(4)}\oplus \mathbb{Z}_{p}^{(5)} &\textrm{if $G \cong$ }~\textrm{$G_{43}.$}\\
   \mathbb{Z}_{p^3}\oplus \mathbb{Z}_{p}^{(5)}  & \textrm{if $G \cong$ }~\textrm{$G_{51}.$}\\
    \mathbb{Z}_{p^2}\oplus \mathbb{Z}_{p}^{(9)}  & \textrm{if $G \cong$ }~\textrm{$G_{66}.$}\\
    \mathbb{Z}_{p}^{(15)}  & \textrm{if $G \cong$ }~\textrm{$G_{70}.$}\\
\end{array} \right.  \]
\end{lem}
 \begin{proof}
Using Theorem \ref{even}, it is trivial.
\end{proof}
 Now we give the structure of $\triangledown(G)$ when $G$ is a non-abelian  group of order dividing $p^5.$
\begin{lem}\label{nabla}
 Let $G$ be a non-abelian  group of order $p^5$, where $p>3$ is an odd prime. Then
  \[\triangledown(G)\cong\left \{ \begin{array}{ll}
     \mathbb{Z}_{p}^{(3)} & \textrm{if $G \cong$ }~\textrm{$G_{3},\dots, G_{12_{k}}, $}~\textrm{$G_{28},\dots, G_{33_{b}} .$}\\
      \mathbb{Z}_{p}^{(6)} & \textrm{if $G \cong$ }~\textrm{$G_{34},$}~\textrm{$G_{36},$}~\textrm{$G_{38,}$}~\textrm{$G_{39},$}~\textrm{$G_{41},$}~\textrm{$G_{42},$}~\textrm{$G_{44},$}~\textrm{$G_{46},$}\\ &~\textrm{$G_{48_{k}}, \dots, G_{50_{k}}, $}~\textrm{$G_{54}, \dots, G_{63}. $}
      \\ \mathbb{Z}_{p}^{(10)} & \textrm{if $G \cong$ }~\textrm{$G_{64},$}~\textrm{$G_{65},$}~\textrm{$G_{67}, \dots, G_{69}. $}\\
       \mathbb{Z}_{p^2}^{(3)}  & \textrm{if $G \cong$ }~\textrm{$G_{2},$}~\textrm{$G_{14}.$}\\
   \mathbb{Z}_{p^2}\oplus\mathbb{Z}_p^{(5)} &\textrm{if $G \cong$ }~\textrm{$G_{35},$}~\textrm{$G_{37},$}~\textrm{$G_{40},$}~\textrm{$G_{45},$}~\textrm{$G_{47},$}~\textrm{$G_{52},$}~\textrm{$G_{53}.$}\\
   \mathbb{Z}_{p^3}\oplus \mathbb{Z}_{p}^{(2)}  & \textrm{if $G \cong$ }~\textrm{$G_{15},$}~\textrm{$G_{25},$}~\textrm{$G_{27}.$}\\
\end{array} \right.  \]
\end{lem}
\begin{proof}
Let $G$ be isomorphic to one of the groups  $G_{3}, \dots, G_{12_{k}}, G_{28}, \dots, G_{33_{b}}$. Since $G^{ab}\cong\mathbb{Z}_{p}^{(2)}$, clearly using Theorems  \ref{2.1} and \ref{even} we have $\triangledown(G)\cong\Gamma(\mathbb{Z}_{p}^{(2)})\cong\Gamma(\mathbb{Z}_{p})\oplus\Gamma(\mathbb{Z}_{p})\oplus(\mathbb{Z}_{p}\otimes \mathbb{Z}_{p})\cong\mathbb{Z}_{p}^{(3)}$. The proof of the rest of groups is similar.
\end{proof}
\section{third homotopy of all groups of order dividing $p^5$}
Brown and Loday in \cite{Brown} described the role of $J_{2}(G)$ in algebraic topology, they showed the third homotopy group of suspension of an Eilenberg-MacLain space $k(G,1) $ satisfied the condition $\Pi_{3}(SK(G,1))\cong J_{2}(G)$. This section is devoted to obtain the structure of third homotopy groups for groups of order dividing $p^5$.
Blyth et al. \cite{Blyth} proved the following theorem that helps us to obtain $J_{2}(G)$ of $p$-groups of order $p^5.$
\begin{thm}\cite[Corollary 2.3]{Blyth}\label{3.1}
Let $G$ be a group such that $G^{ab}$ is a finitely generated abelian group with
no elements of order $2$. Then $J_{2}(G)\cong \Gamma (G^{ab})\times \mathcal{M}(G).$
\end{thm}
 First we give the structure of $J_{2}(G)$ when $G$ is an abelian  group of order dividing $p^5$.

\begin{lem}\label{100}
 Let $G$ be an abelian  group of order dividing $p^5$, where $p$  $(p>3)$ is an odd prime. Then
  \[J_{2}(G)\cong\left \{ \begin{array}{ll}
     \mathbb{Z}_{p^5} & \textrm{if $G \cong$ }~\textrm{$G_{1}.$}\\
      \mathbb{Z}_{p^3}\oplus\mathbb{Z}_{p^2}^{(3)} & \textrm{if $G \cong$ }~\textrm{$G_{13}.$}\\ \mathbb{Z}_{p^4}\oplus\mathbb{Z}_{p}^{(3)} & \textrm{if $G \cong$ }~\textrm{$G_{26}.$}\\
    \mathbb{Z}_{p^2}^{(4)}\oplus \mathbb{Z}_{p}^{(5)} &\textrm{if $G \cong$ }~\textrm{$G_{43}.$}\\
   \mathbb{Z}_{p^3}\oplus \mathbb{Z}_{p}^{(8)}  & \textrm{if $G \cong$ }~\textrm{$G_{51}.$}\\
    \mathbb{Z}_{p^2}\oplus \mathbb{Z}_{p}^{(15)}  & \textrm{if $G \cong$ }~\textrm{$G_{66}.$}\\
    \mathbb{Z}_{p}^{(25)}  & \textrm{if $G \cong$ }~\textrm{$G_{70}.$}\\
\end{array} \right.  \]
\end{lem}
\begin{proof}
Using Theorems \ref{3.1}, \ref{Schur} and \ref{nabla}, it is trivial.
\end{proof}

\begin{lem}\label{101}
 Let $G$ be a non-abelian  group of order dividing $p^5$, where $p$  $(p>3)$ is an odd prime. Then
  \[J_{2}(G)\cong \left \{ \begin{array}{ll}
     \mathbb{Z}_{p}^{(3)} & \textrm{if $G \cong$ }~\textrm{$G_{7}, \dots, G_{9},$}~\textrm{$G_{11,1},$}~\textrm{$G_{12_{k}}.$}\\
     \mathbb{Z}_{p}^{(4)} & \textrm{if $G \cong$ }~\textrm{$G_{4}, \dots, G_{6}$}~\textrm{$G_{10},$}~\textrm{$G_{11,2},$}~\textrm{$G_{29_{a}},$}~\textrm{$G_{11,1},$}~\textrm{$G_{30},$}~\textrm{${32_{a}},$}~\textrm{or $G_{33_{b}}.$}\\  \mathbb{Z}_{p}^{(7)} & \textrm{if $G \cong$}~\textrm{$G_{42},$}~\textrm{$G_{44},$}~\textrm{$G_{46},$}~\textrm{$G_{48,1},$}~\textrm{$G_{50_{k}}.$}\\
     \mathbb{Z}_{p}^{(9)} & \textrm{if $G \cong$}~\textrm{$G_{36},$}~\textrm{$G_{38},$}~\textrm{$G_{39},$}~\textrm{$G_{55}, \dots, G_{58},$}~\textrm{$G_{60}, \dots, G_{63}.$}\\
     \mathbb{Z}_{p}^{(10)} & \textrm{if $G \cong$}~\textrm{$G_{54},$}~\textrm{$G_{59}.$}\\
     \mathbb{Z}_{p}^{(12)} & \textrm{if $G \cong$}~\textrm{$G_{34}.$}\\
     \mathbb{Z}_{p}^{(15)} & \textrm{if $G \cong$}~\textrm{$G_{65},$}~\textrm{$G_{67}, \dots, G_{69}.$}\\
     \mathbb{Z}_{p}^{(17)} & \textrm{if $G \cong$}~~\textrm{$G_{64}.$}\\
      \mathbb{Z}_{p^2}\oplus\mathbb{Z}_{p}^{(2)} & \textrm{if $G \cong$}~\textrm{$G_{24}.$}\\
       \mathbb{Z}_{p^2}\oplus\mathbb{Z}_{p}^{(3)}& \textrm{if $G \cong$}~\textrm{$G_{19}.$}\\
         \mathbb{Z}_{p^2}\oplus\mathbb{Z}_{p}^{(4)} & \textrm{if $G \cong$}~\textrm{$G_{16}, \dots, G_{18},$}~\textrm{$G_{20}, \dots, G_{22}.$}\\
           \mathbb{Z}_{p^2}\oplus\mathbb{Z}_{p}^{(6)} & \textrm{if $G \cong$}~\textrm{$G_{45},$}~\textrm{$G_{48,2},$}~\textrm{$G_{49}.$}\\
            \mathbb{Z}_{p^2}\oplus\mathbb{Z}_{p}^{(8)} & \textrm{if $G \cong$}~\textrm{$G_{40}.$}\\
             \mathbb{Z}_{p^2}\oplus\mathbb{Z}_{p}^{(9)} & \textrm{if $G \cong$}~\textrm{$G_{35},$}~\textrm{$G_{37}.$}\\
    \mathbb{Z}_{p^2}^{(3)}\oplus\mathbb{Z}_{p}& \textrm{if $G \cong$}~\textrm{$G_{14}.$}\\
     \mathbb{Z}_{p^2}^{(3)}\oplus\mathbb{Z}_{p}^{(3)}& \textrm{if $G \cong$}~\textrm{$G_{2}.$}\\
    \mathbb{Z}_{p^3}\oplus\mathbb{Z}_{p}^{(3)} & \textrm{if $G \cong$}~\textrm{$G_{25},$}~\textrm{$G_{27}.$}\\
   \end{array} \right.  \]
\end{lem}
\begin{proof}
 Using Theorems \ref{3.1}, \ref{Schur} and \ref{nabla}, we have the result.
\end{proof}
\section{non-abelian tensor square and non-abelian exterior square of all groups of order dividing $p^5$}

This section illustrates our main results, our goal is to obtain non-abelian exterior square and non-abelian tensor square of all groups of order dividing $p^5$.
The following theorem is  a key tool to obtain the structure of $G\wedge G$ .
\begin{thm}\cite[Proposition 16 (iv)]{Ellis5}\label{weg}
Let $G$ be a group and $N\unlhd G$. Then $G/N\wedge G/N \cong G\wedge G$ if and only if $N\leq Z^{\wedge}(G).$
\end{thm}

\begin{thm}\label{Wegabelian}
 Let $G$ be an abelian groups of order dividing $p^5$. Then
  \[G\wedge G\cong \left \{ \begin{array}{ll}
  $1$ & \textrm{if $G \cong$ }~\textrm{$G_{1}.$}\\ \mathbb{Z}_{p^2} & \textrm{if $G \cong$ }~\textrm{$G_{13}.$}\\
 \mathbb{Z}_{p} & \textrm{if $G \cong$ }~\textrm{$G_{26}.$}\\
 \mathbb{Z}_{p^2} \oplus \mathbb{Z}_{p} ^{(2)} & \textrm{if $G \cong$ }~\textrm{$G_{43}.$}\\
  \mathbb{Z}_{p} ^{(3)} & \textrm{if $G \cong$ }~\textrm{$G_{51}.$}\\ \mathbb{Z}_{p} ^{(6)} & \textrm{if $G \cong$ }~\textrm{$G_{66}.$}\\
   \mathbb{Z}_{p} ^{(10)} & \textrm{if $G \cong$ }~\textrm{$G_{70}.$}\\
\end{array} \right.  \]
\end{thm}
\begin{proof}
Since the derived subgroup of an abelian group $G$ is trivial, we have $G\wedge G\cong \mathcal{M}(G).$ Now the result follows by Theorem \ref{Schur}.
\end{proof}
Our main theorem in this context is the following.
\begin{thm}\label{Weg}
 Let $G$ be a  non-abelian group of order dividing $p^5$, where $p (p>3)$ is an odd prime. Then
  \[G \wedge G\cong \left \{ \begin{array}{ll}
   \mathbb{Z}_{p} & \textrm{if $G \cong$ }~\textrm{$G_{27}.$}\\    \mathbb{Z}_{p}^{(3)} & \textrm{if $G \cong$ }~\textrm{$G_{7}, \dots, G_{9},$}~\textrm{$G_{11,1},$}~\textrm{$G_{12_{k}},$}~\textrm{$G_{15},$}~\textrm{$G_{19},$}~\textrm{$G_{21}, \dots, G_{23},$}~\textrm{$G_{42},$}\\ &~\textrm{$G_{44},$}~\textrm{$G_{46},$}~\textrm{$G_{47},$}~\textrm{$G_{481},$}~\textrm{$G_{50_{k}},$}~\textrm{$G_{52},$}~\textrm{$G_{53}.$}\\
 \mathbb{Z}_{p}^{(4)} & \textrm{if $G \cong$ }~\textrm{$G_{4}, \dots, G_{6},$}~\textrm{$G_{16},$}~\textrm{$G_{20},$}~\textrm{$G_{29_{a}},$}~\textrm{$G_{30},$}~\textrm{$G_{32_{a}},$}~\textrm{$G_{33_{b}}.$}\\
  \mathbb{Z}_{p}^{(5)} & \textrm{if $G \cong$ }~\textrm{$G_{35}, \dots, G_{39},$}~\textrm{$G_{55}, \dots, G_{58},$}~\textrm{$G_{60}, \dots, G_{63}.$}\\
 \mathbb{Z}_{p}^{(6)} & \textrm{if $G \cong$ }~\textrm{$G_{3},$}~\textrm{$G_{28},$}~\textrm{$G_{31},$}~\textrm{$G_{54},$}~\textrm{$G_{59},$}~\textrm{$G_{65},$}~\textrm{$G_{67}, \dots, G_{69}.$}\\
  E_{1}\times\mathbb{Z}_{p}^{(3)} & \textrm{if $G \cong$ }~\textrm{$G_{28},$}~\textrm{$G_{31}.$}\\
 \mathbb{Z}_{p}^{(8)} & \textrm{if $G \cong$ }~\textrm{$G_{34},$}~\textrm{$G_{64}.$}\\
  \mathbb{Z}_{p^2} & \textrm{if $G \cong$ }~\textrm{$G_{14},$}~\textrm{$G_{24},$}~\textrm{$G_{25}.$}\\
   \mathbb{Z}_{p^2}\oplus \mathbb{Z}_{p}^{(2)} & \textrm{if $G \cong$ }~\textrm{$G_{2},$}~\textrm{$G_{10},$}~\textrm{$G_{11,2},$}~\textrm{$G_{17},$}~\textrm{$G_{18},$}~\textrm{$G_{40},$}~\textrm{$G_{41},$}~\textrm{$G_{45},$}~\textrm{$G_{48,2},$}~\textrm{$G_{49}.$}\\
  \end{array} \right.  \]
\end{thm}
\begin{proof}
\begin{itemize}
\item[($i.$)] Let $G\cong G_{7}, \dots, G_{9},G_{11,1},G_{12_{k}},G_{24}$ or $G_{27}.$ \newline Since Theorem \ref{Schur} implies $\mathcal{M}(G)=1,$ we have $G\wedge G\cong G'$ and the result follows.
\item[($ii.$)]Let $G\cong G_{4}, G_{5}, G_{6},G_{16},G_{20} , G_{29_{a}}, G_{30}, G_{32_{a}}$ or $G_{33_{b}}.$ \newline Using Theorem \ref{weg} and putting $N=Z^{\wedge}(G)=\mathbb{Z}_{p},$ we have
$G\wedge G\cong G/N\wedge G/N$, now using Table 3. of \cite{Ghorbanzadeh}, we have $G\wedge G\cong\mathbb{Z}_{p}^{(4)}.$
\item[$(iii).$] Let $ G\cong G_{3}$.\newline This group is a capable group and  we know $\vert G\wedge G\vert=\vert \mathcal{M}(G) \vert \vert (G') \vert$, so by using Theorem \ref{Schur}
 , we have $\vert G\wedge G\vert=p^{6}. $ Using \cite [Theorem 2]{Moravec}, the exponent of $G\otimes G$ is $p$, since the exponent of $G$ is $p$. Therefore the exponent of $G\wedge G$ is $p$.\newline On the other hand $G\wedge G$ is an abelian group.\newline We have $(G\wedge G)'=\langle [x,y] \wedge [x',y'] \vert x, y, x', y' \in G\rangle=\langle [(x \wedge y) , (x' \wedge y')] \vert x, y, x', y' \in G\rangle=1$,\newline since $G'=\langle [g_{i}, g_{j}] \vert i, j \in \lbrace 3, 4, 5 \rbrace \rangle$ it is enough to prove that \newline
 $(g_{3}\wedge g_{4})=1, (g_{3}\wedge g_{5})=1$ and $(g_{4}\wedge g_{5})=1$.\newline Hence $G\wedge G$ is an elementary abelian group and then $ G\wedge G\cong \mathbb{Z}_{p}^{(6)}.$
 \item[$(iv).$] Let $ G\cong G_{34}, G_{54},G_{59}$ or $G_{64}.$
 \newline The proof  for this groups is completely similar to $(iii).$
\item[$(v).$]  Let $ G\cong G_{2}, G_{10}, G_{11,2}, G_{17}, G_{18},  G_{28}, G_{31}, G_{40}, G_{41}, G_{45}, G_{48,2}$ or $ G_{49}.$ \newline
This groups are also capable and we use the result of \cite{Hatui}.
Following notation and terminology of \cite{Hatui}, $G_{2}=\phi_{2}(221)d, G_{10}=\phi_{6}(221)d_{0},$\newline$ G_{11,2}=\phi_{6}(221)b_{1/2}(p-1), G_{17}=\phi_{3}(221)b_{r}, G_{18}=\phi_{3}(221)b_{r}, $\newline$G_{28}=\phi_{9}(1^5), G_{31}=\phi_{10}(1^5), G_{40}=\phi_{2}(221)a, G_{41}=\phi_{4}(221)b, $\newline$G_{45}=\phi_{2}(221)c, G_{48,2}=\phi_{4}(221)d_{1/2}(p-1)
$ and $G_{49}=\phi_{4}(221)f_{0}$.
$\newline$Here $ E_{1}$ denotes the non-abelian group of order $p^3$  and exponent $p$.
\item[$(vi).$] The proof  for the rest of groups is completely similar to $(ii).$
\end{itemize}
\end{proof}
The next theorem states the structure of the tensor square of groups with respect to the direct products of two groups.
\begin{thm}\cite[Proposition 11]{Brown}
 Let $G$ and $H$ be groups. Then $(G\times H)\otimes (G\times H) \cong (G\otimes G)\times (G^{ab}\otimes H^{ab})\times (H^{ab}\otimes G^{ab})\times (H\otimes H)$.
\end{thm}
Blyth et al. in \cite{Blyth} give us the structure of the non-abelian tensor square of a group $G$ with $G^{ab}$  finitely generated as follows.
\begin{thm}\cite[Theorem 1]{Blyth}\label{tensorsquare}
 Let $G$ be a group such that $G^{ab}$ is finitely generated. If  $G^{ab}$ has no element of order $2$ or if $G'$ has a complement in $G$, then $G\otimes G\cong\Gamma(G^{ab})\times G \wedge G $.
\end{thm}
\begin{thm}\label{Wegabelian}
 Let $G$ be an abelian group of order dividing $p^5$. Then
  \[G\otimes G\cong \left \{ \begin{array}{ll}
   \mathbb{Z}_{p^5} & \textrm{if $G \cong$ }~\textrm{$G_{1}.$}\\ \mathbb{Z}_{p^3}\oplus \mathbb{Z}_{p^2} ^{(3)}  & \textrm{if $G \cong$ }~\textrm{$G_{13}.$}\\
 \mathbb{Z}_{p^4} \oplus \mathbb{Z}_{p} ^{(3)} & \textrm{if $G \cong$ }~\textrm{$G_{26}.$}\\
 \mathbb{Z}_{p^2} ^{(4)}\oplus \mathbb{Z}_{p} ^{(5)} & \textrm{if $G \cong$ }~\textrm{$G_{43}.$}\\
  \mathbb{Z}_{p^3} \oplus\mathbb{Z}_{p}^{(8)} & \textrm{if $G \cong$ }~\textrm{$G_{51}.$}\\  \mathbb{Z}_{p^2} \oplus\mathbb{Z}_{p}^{(15)}  & \textrm{if $G \cong$ }~\textrm{$G_{66}.$}\\
   \mathbb{Z}_{p} ^{(25)} & \textrm{if $G \cong$ }~\textrm{$G_{70}.$}\\
\end{array} \right.  \]
\end{thm}
\begin{proof}
Using Theorems \ref{tensorsquare}, \ref{Weg} and \ref{nabla}, it is trivial.
\end{proof}
\begin{thm}\label{tensor}
 Let $G$ be a non-abelian group of order dividing $p^5$ where $p (p>3)$ is an odd prime. Then
  \[G \otimes G\cong \left \{ \begin{array}{ll}
   \mathbb{Z}_{p}^{(6)} & \textrm{if $G \cong$ }~\textrm{$G_{7}, \dots, G_{9},$}~\textrm{$G_{11,1},$}~\textrm{or $G_{12_{k}}.$}
\\   \mathbb{Z}_{p}^{(7)} & \textrm{if $G \cong$ }~\textrm{$G_{4}, \dots, G_{6},$}~\textrm{$G_{11,1},$}~\textrm{$G_{29_{a}},$}~\textrm{$G_{30},$}~\textrm{$G_{32_{a}},$}~\textrm{$G_{33_{b}}.$}\\
E_{1}\times \mathbb{Z}_{p}^{(6)} & \textrm{if $G \cong$ }~\textrm{$G_{28}$}~\textrm{or $G_{31}.$}\\
 \mathbb{Z}_{p}^{(9)}& \textrm{if $G \cong$ }~\textrm{$G_{3},$}~\textrm{$G_{42},$}~\textrm{$G_{44},$}~\textrm{$G_{46},$}~\textrm{$G_{48,1},$}~\textrm{$G_{50_{k}}.$}\\
 \mathbb{Z}_{p}^{(11)}& \textrm{if $G \cong$ }~\textrm{$G_{36},$}~\textrm{$G_{38},$}~\textrm{$G_{39},$}~\textrm{$G_{55}, \dots, G_{58},$}~\textrm{$G_{60}, \dots, G_{63}.$}\\
  \mathbb{Z}_{p}^{(12)}& \textrm{if $G \cong$ }~\textrm{$G_{54},$}~\textrm{$G_{59}.$}\\
   \mathbb{Z}_{p}^{(14)}& \textrm{if $G \cong$ }~\textrm{$G_{34}.$}\\
\mathbb{Z}_{p}^{(16)}& \textrm{if $G \cong$ }~\textrm{$G_{65}, \dots, G_{69}.$}\\
 \mathbb{Z}_{p}^{(18)}& \textrm{if $G \cong$ }~\textrm{$G_{64}.$}\\
  \mathbb{Z}_{p^2}^{(4)}& \textrm{if $G \cong$ }~\textrm{$G_{14}.$}\\
   \mathbb{Z}_{p^2}^{(2)}\oplus  \mathbb{Z}_{p}^{(2)}& \textrm{if $G \cong$ }~\textrm{$G_{24}.$}\\
   \mathbb{Z}_{p^2}^{(4)}\oplus  \mathbb{Z}_{p}^{(2)}& \textrm{if $G \cong$ }~\textrm{$G_{2}.$}\\
   \mathbb{Z}_{p^2}^{(2)}\oplus  \mathbb{Z}_{p}^{(4)}& \textrm{if $G \cong$ }~\textrm{$G_{17},$}~\textrm{$G_{18}.$}\\
    \mathbb{Z}_{p^2}^{(2)}\oplus  \mathbb{Z}_{p}^{(5)}& \textrm{if $G \cong$ }~\textrm{$G_{10},$}~\textrm{$G_{11,2},$}~\textrm{$G_{19},$}~\textrm{$G_{21}, \dots, G_{23}.$}\\
    \mathbb{Z}_{p^2}\oplus  \mathbb{Z}_{p}^{(6)}& \textrm{if $G \cong$ }~\textrm{$G_{16},$}~\textrm{$G_{20}.$}\\
     \mathbb{Z}_{p^2}\oplus  \mathbb{Z}_{p}^{(8)}& \textrm{if $G \cong$ }~\textrm{$G_{41},$}~\textrm{$G_{47},$}~\textrm{$G_{48,2},$}~\textrm{$G_{49},$}~\textrm{$G_{52},$}~\textrm{$G_{53}.$}\\
     \mathbb{Z}_{p^2}\oplus  \mathbb{Z}_{p}^{(10)}& \textrm{if $G \cong$ }~\textrm{$G_{35},$}~\textrm{$G_{37}.$}\\
      \mathbb{Z}_{p^2}^{(2)}\oplus  \mathbb{Z}_{p}^{(7)}& \textrm{if $G \cong$ }~\textrm{$G_{40},$}~\textrm{$G_{45}.$}\\
      \mathbb{Z}_{p^3}\oplus  \mathbb{Z}_{p}^{(3)}& \textrm{if $G \cong$ }~\textrm{$G_{27}.$}\\
      \mathbb{Z}_{p^3}\oplus  \mathbb{Z}_{p}^{(5)}& \textrm{if $G \cong$ }~\textrm{$G_{15}.$}\\
      \mathbb{Z}_{p^2}\oplus \mathbb{Z}_{p^3}\oplus  \mathbb{Z}_{p}^{(2)}& \textrm{if $G \cong$ }~\textrm{$G_{25}.$}\\
\end{array} \right.  \]
\end{thm}

\begin{proof}
Using Theorems \ref{tensorsquare}, \ref{Weg} and \ref{nabla}, we have the result.
\end{proof}
\section{tensor center and exterior center of all groups of order dividing $p^5$}
The purpose of this section is to obtain the structure of $Z^{\otimes}(G)$ and $Z^{\wedge}(G)$ when $G$ is a group of order dividing $p^5$.
In the following theorem the tensor center of an arbitrary finite abelian group is determined.
\begin{lem} \cite[Proposition 1.8]{Ellis5}\label{trivial}
Let $G$ be a finite abelian p-group of order $p^5$. Then  $Z^{\otimes}(G)=1.$
\end{lem}
Following theorem is a key tool for the next investigations.
\begin{thm}\cite[Proposition 16 (v)]{Ellis5}\label{tensor center}
Let $G$ be a group and  $N\unlhd G$. Then $G/N\otimes G/N \cong G\otimes G$ if and only if $N\leq Z^{\otimes}(G).$
\end{thm}
The following corollary is a straightforward consequence of Lemma \ref{2.1} and Theorem \ref{1.6}.
\begin{cor}
 \label{epicenter}
 Let $G$ be a group of order dividing $p^5$ where $p$ $ (p>3)$ is an odd prime. Then
  \[Z^{\wedge}(G)\cong\left \{ \begin{array}{ll}
     1 & \textrm{if $G \cong$ }~\textrm{$G_{2},$}~\textrm{$G_{3}$}~\textrm{$G_{10},$}~\textrm{$G_{11,2},$}~\textrm{$G_{17},$}~\textrm{$G_{18},$}~\textrm{$G_{28},$}~\textrm{$G_{31},$}~\textrm{$G_{34},$}~\textrm{$G_{40},$}~\textrm{$G_{41},$}\\ &~\textrm{$G_{43},$}~\textrm{$G_{45},$}~\textrm{$G_{48,2},$}~\textrm{$G_{49},$}~\textrm{$G_{54},$}~\textrm{$G_{59},$}~\textrm{$G_{17},$}~\textrm{$G_{64},$}~\textrm{$G_{70}.$}\\
      \mathbb{Z}_{p} & \textrm{if $G \cong$ }~\textrm{$G_{4}, \dots, G_{6}, $}~\textrm{$G_{10},$}~\textrm{$G_{14},$}~\textrm{$G_{16},$}~\textrm{$G_{25},$}~\textrm{$G_{29_{a}},$}~\textrm{$G_{30},$}~\textrm{$G_{32_{a}},$}~\textrm{$G_{33_{b}},$}~\textrm{$G_{35},$}\\ &~\textrm{$G_{36}, \dots, G_{39}, $}~\textrm{$G_{55}, \dots, G_{58}, $}~\textrm{$G_{61}, \dots, G_{63}, $}~\textrm{$G_{65}, \dots, G_{69}. $}
      \\ \mathbb{Z}_{p}^{(2)} & \textrm{if $G \cong$ }~\textrm{$G_{7}, \dots, G_{9}, $}~\textrm{$G_{11,1},$}~\textrm{$G_{12_{k}},$}~\textrm{$G_{21}, $}~\textrm{$G_{42}, $}~\textrm{$G_{44}, $}~\textrm{$G_{46}, $}~\textrm{$G_{47}, $}~\textrm{$G_{48,1}, $}\\ &~\textrm{$G_{50_{k}}.$}\\
      \mathbb{Z}_{p^2}  & \textrm{if $G \cong$ }~\textrm{$G_{15},$}~\textrm{$G_{20}, $}~\textrm{$G_{22}, $}~\textrm{$G_{23}, $}~\textrm{$G_{51}, \dots, G_{53}. $}\\
       \mathbb{Z}_{p^3} & \textrm{if $G \cong$ }~\textrm{$G_{26}, $}~\textrm{$G_{27}.$}\\
       \mathbb{Z}_{p^5}  & \textrm{if $G \cong$ }~\textrm{$G_{70}.$}\\
\end{array} \right.  \]
 \end{cor}
In next theorem we obtain tensor center of all groups of order dividing $p^5.$
\begin{thm}
 Let $G$ be a non-abelian group of order dividing $p^5$ where $p$ $ (p>3)$ is prime. Then
  \[Z^{\otimes}(G)\cong \left \{ \begin{array}{ll}
     $1$ & \textrm{if $G \cong$ }~\textrm{$G_{1}, \dots, G_{3},$}~\textrm{$G_{10},$}~\textrm{$G_{11,2},$}\textrm{$G_{13},$}~\textrm{$G_{15}, \dots, G_{18},$}~\textrm{$G_{20},$}~\textrm{$G_{25},$}~\textrm{$G_{26},$}\\ & ~\textrm{$G_{28},$}~\textrm{$G_{31},$}~\textrm{$G_{34},$}~\textrm{$G_{35},$}~\textrm{$G_{37},$}~\textrm{$G_{40},$}~\textrm{$G_{41},$}~\textrm{$G_{45},$}~\textrm{$G_{48,2},$}~\textrm{$G_{49},$}~\textrm{$G_{51},$}~\textrm{$G_{54},$}\\ &~\textrm{$G_{59},$}~\textrm{$G_{64},$}~\textrm{$G_{66},$}~\textrm{$G_{70}.$}\\
        \mathbb{Z}_{p} & \textrm{$if $ $G \cong$ }~\textrm{$G_{4}, \dots, G_{6},$}~\textrm{$G_{14},$}~\textrm{$G_{19},$}~\textrm{$G_{21}, \dots, G_{24},$}~\textrm{$G_{27},$}~\textrm{$G_{29_{a}},$}~\textrm{$G_{30},$}~\textrm{$G_{32_{a}},$}\\ &~\textrm{$G_{33_{b}},$}~\textrm{$G_{36},$}~\textrm{$G_{38},$}~\textrm{$G_{39},$}~\textrm{$G_{52},$}~\textrm{$G_{53},$}~\textrm{$G_{55}, \dots, G_{58}$}~\textrm{$G_{60}, \dots, G_{63},$}~\textrm{$G_{67}, \dots, G_{69}.$}\\
  \mathbb{Z}_{p}^{(2)}& \textrm{if $G \cong$ }~\textrm{$G_{7}, \dots, G_{9},$}~\textrm{$G_{11,1},$}~\textrm{$G_{12_{k}},$}~\textrm{$G_{42},$}~\textrm{$G_{44},$}~\textrm{$G_{46},$}~\textrm{$G_{48,1},$}~\textrm{$G_{50_{k}}.$}\\
\end{array} \right.  \]
\end{thm}
\begin{proof}
\begin{itemize}
\item[$(i).$] Let $G$ be isomorphic to one of groups  $G_{2}, G_{3},G_{10}, G_{11,2}, G_{17}, G_{18},$\\$  G_{28}, G_{31}G_{34}, G_{40}, G_{41}, G_{43}, G_{45}, G_{48,2}, G_{49}, G_{59}, G_{64}, G_{18}$ or $G_{70}$.\newline Clearly by using Lemmas \ref{88} and \ref{epicenter}, $G$ is  capable and $Z^{\otimes}(G) \leq Z^{\wedge}(G)=Z^{*}(G)\cong 1$, so $Z^{\otimes}(G)\cong 1$.
\item[$(ii).$]
 Let $G\cong  G_{1}, G_{13}, G_{26}, G_{43}, G_{51}, G_{66}$ or $G_{70}.$ By Theorem \ref{trivial}, $Z^{\otimes}(G)\cong 1.$
 \item[$(iii).$]
 Let $G\cong G_{4}$. Using Theorem \ref{tensor center} and put $N \cong \mathbb{Z}_{p}$, using Table $3.$ of \cite{Ghorbanzadeh}, we have  $G/N\otimes G/N \cong \mathbb{Z}_{p}^{(7)}.$ On the other hand by Theorem \ref{tensor}, since $G\otimes G\cong \mathbb{Z}_{p}^{(7)}$, we have  $ \mathbb{Z}_{p} \leq Z^{\otimes}(G).$ Again using Lemmas \ref{88} and \ref{epicenter}, we have  $Z^{\otimes}(G)\cong \mathbb{Z}_{p}$.
 \item[$(iv).$] Let $G$ be isomorphic to one of groups  $G_{5}, G_{6}, G_{14}, G_{24}, G_{29_{a}}, G_{30}, G_{33_{b}}, G_{34},$\newline$ G_{36}, G_{38}, G_{39}, G_{55}, G_{56}, G_{57}, G_{58}, G_{60}, G_{61}, G_{62}, G_{63}, G_{65}, G_{67}, G_{68} $ or $G_{69}$. The proof is completely similar to $(iii).$
 \item[$(v).$]
 Let $G\cong G_{7}$. Using Theorem \ref{tensor center} and put $N \cong \mathbb{Z}_{p}^{(2)},$ the Table $3.$ of \cite{Ghorbanzadeh} shows $G/N\otimes G/N \cong \mathbb{Z}_{p}^{(6)}.$ On the other hand, by Theorem \ref{tensor}, since $G\otimes G\cong \mathbb{Z}_{p}^{(6)}$, $   \mathbb{Z}_{p}^{(2)} \leq Z^{\otimes}(G).$ Again using Lemmas \ref{88} and \ref{epicenter}, we have  $Z^{\otimes}(G)\cong \mathbb{Z}_{p}^{(2)}$.
 \item[$(vi).$] Let $G$ be isomorphic to one of groups  $G_{8}, G_{9}, G_{11,1}, G_{42}, G_{44}, G_{46}, G_{48,1} $ or $ G_{50_{k}}$. The proof  is completely similar to $(v).$
 \item[$(vii).$] Let $G$ be isomorphic to one of the groups  $G_{15}, G_{16}, G_{20}, G_{35} $ or $G_{37}$. \newline There is no normal subgroup $N$ such that $G\otimes G\cong G/N\otimes G/N $, so $Z^{\otimes}(G)\cong 1$.
\end{itemize}
\end{proof}

 Here we summarize the results of the paper in the next tables. In the first table $cl(G)($ the nilpotency class of $G)$, $\mathcal{M}(G), Z(G), G', G^{ab} (G/G'), \triangledown(G)$, $J_{2}(G)$ and in the second table $cl(G), \mathcal{M}(G), G\wedge G, Z^{\wedge}(G), G\otimes G, {Z}^{\otimes}(G)$  of all groups of order dividing $p^5$ are given.
\begin{table}[h!]
\centering
\caption{Fig. 1.}
\begin{tabular}{|c|c|c|c|c|c|c|c|}
 \hline
\textbf{Type of} \large{$G$} &\large{$cl(G)$} &\large{$\mathcal{M}(G)$} &\large{$Z(G)$}&\large{ $G'$}&\large{$G^{ab}$}&\large{$\bigtriangledown(G)$} &\large{$J_{2}(G)$}\\
\hline
\hline
$G_{1}$  &$1$ &$1$& $\mathbb{Z}_{p^5 }$ & $1$ & $ \mathbb{Z}_{p^5 }$ & $ \mathbb{Z}_{p^5 }$& $\mathbb{Z}_{p^5 }$\\
 \hline
$G_{2}$ &$2$ & $\mathbb{Z}_{p}^{(3)}$& $\mathbb{Z}_{p}^{(3)}$ &$ \mathbb Z_{p}$ & $\mathbb Z_{p^{2}}^{(2)}$ & $\mathbb Z_{p^{2}}^{(3)}$  &$ \mathbb Z_{p}^{(3)} \oplus \mathbb Z_{p^2}^{(3)}$\\
\hline
 $G_{3}$  &$3$  & $\mathbb{Z}_{p}^{(3)}$&$ \mathbb{Z}_{p}^{(2) }$ &$\mathbb Z_{p}^{(3)}$ &  $\mathbb Z_{p}^{(2)}$ & $\mathbb Z_{p}^{(3)}$&$\mathbb Z_{p}^{(6)}$\\
\hline
$G_{4}$  &$3$  & $ \mathbb{Z}_{p}$&$ \mathbb{Z}_p^{(2)}$ &$ \mathbb{Z}_p^{(3)} $ &  $\mathbb Z_{p}^{(2)}$ & $\mathbb Z_{p}^{(3)}$&$\mathbb Z_{p}^{(4)} $\\
\hline
 $G_{5}$ &$3$    & $\mathbb{Z}_{p}$&$\mathbb{Z}_p^{(2)}$ &$\mathbb{Z}_p^{(3)}$ &  $\mathbb Z_{p}^{(2)}$ & $\mathbb Z_{p}^{(3)}$ &$\mathbb Z_{p}^{(4)} $\\
 \hline
 $G_{6}$  &$3$ &$\mathbb{Z}_{p}$&$\mathbb{Z}_p^{(2)}$ &$\mathbb{Z}_p^{(3)}$ &  $\mathbb Z_{p}^{(2)}$ & $\mathbb Z_{p}^{(3)}$ &$\mathbb Z_{p}^{(4)} $ \\
\hline
 $G_{7}$  &$3$&$1$& $\mathbb{Z}_p^{(2)}$ &$\mathbb{Z}_p^{(3)}$ &  $\mathbb Z_{p}^{(2)}$ & $\mathbb Z_{p}^{(3)}$ &$\mathbb Z_{p}^{(3)} $ \\
\hline
$G_{8}$ &$3$  &$1$& $\mathbb{Z}_p^{(2)}$ &$ \mathbb{Z}_p^{(3)}$ &  $\mathbb Z_{p}^{(2)}$& $\mathbb Z_{p}^{(3)}$&$\mathbb Z_{p}^{(3)} $ \\
\hline
 $G_{9}$   &$3$ &$1$&$\mathbb{Z}_p^{(2)}$ &$\mathbb{Z}_p^{(3)}$ &  $ \mathbb Z_{p}^{(2)}$ & $ \mathbb Z_{p}^{(3)}$ &$\mathbb Z_{p}^{(3)} $ \\
\hline
 $G_{10} $  &$3$ &$\mathbb{Z}_{p}$ &$\mathbb{Z}_p^{(2)}$ & $\mathbb{Z}_p^{(3)}$ &  $\mathbb Z_{p}^{(2)}$ & $\mathbb Z_{p}^{(3)}$  & $\mathbb Z_{p}^{(4)} $\\
\hline
  \small {$G_{11,1}=G_{11_{k}} $ if $k\neq \frac{p-1}{2}$}  &$3$ & $1$ & $\mathbb{Z}_p^{(2)}$ &$\mathbb{Z}_p^{(3)}$ &  $\mathbb Z_{p}^{(2)}$ &  $\mathbb Z_{p}^{(3)}$ &$\mathbb Z_{p}^{(3)} $ \\
\hline
 \small {$G_{11,2}=G_{11_{k}} $ if $k=\frac{p-1}{2}$} &$3$ &$\mathbb Z_{p}$ & $\mathbb{Z}_p^{(2)}$ &$\mathbb{Z}_p^{(3)}$ &  $\mathbb Z_{p}^{(2)}$ &  $\mathbb Z_{p}^{(3)}$ &$\mathbb Z_{p}^{(4)} $ \\
\hline
 $G_{12_{k}} $  &$3$&$1$&$\mathbb{Z}_p^{(2)}$ &$\mathbb Z_{p}^{(3)}$ &  $\mathbb Z_{p}^{(2)}$ & $\mathbb Z_{p}^{(3)}$& $\mathbb Z_{p}^{(3)} $ \\
\hline
 $G_{13} $  &$1$ & $\mathbb{Z}_{p^2}$&$\mathbb{Z}_{p^2}\oplus \mathbb{Z}_{p^3}$ &$1$& $\mathbb{Z}_{p^3}\oplus \mathbb{Z}_{p^2}$ &  $\mathbb{Z}_{p^3}\oplus \mathbb{Z}_{p^2}^{(2)}$& $\mathbb {Z}_{p^2}^{(3)} \oplus \mathbb Z_{p^3} $ \\
\hline
 $G_{14} $ &$2$&$\mathbb{Z}_{p}$ & $\mathbb{Z}_{p^2}\oplus \mathbb{Z}_{p}$ &$\mathbb Z_{p}$ &  $\mathbb{Z}_{p^2}^{(2)}$ & $\mathbb{Z}_{p^2}^{(3)}$  &$\mathbb  Z_{p^2}^{(3)}\oplus \mathbb{Z}_{p}$\\
\hline
 $G_{15} $ &$2$ &$\mathbb{Z}_{p}^{(2)}$ & $\mathbb{Z}_{p^2}\oplus \mathbb{Z}_{p}$  &  $\mathbb{Z}_{p}$&$\mathbb{Z}_{p^3}\oplus \mathbb{Z}_{p}$ &$\mathbb{Z}_p^{(2)}\oplus\mathbb Z_{p^3} $&$\mathbb{Z}_p^{(4)}\oplus\mathbb Z_{p^3} $ \\
\hline
 $G_{16} $   &$3$ &$\mathbb{Z}_{p}^{(2)}$&  $\mathbb{Z}_{p}^{(2)}$ &$\mathbb{Z}_{p}^{(2)}$&  $\mathbb{Z}_{p^2}\oplus \mathbb{Z}_{p}$& $\mathbb{Z}_{p^2}\oplus \mathbb{Z}_{p}^{(2)}$ &$\mathbb{Z}_p^{(4)}\oplus\mathbb Z_{p^2} $ \\
\hline
 $G_{17} $   &$3$ &$\mathbb{Z}_{p}^{(2)}$& $\mathbb{Z}_{p}^{(2)}$ &$\mathbb{Z}_{p}^{(2)}$  & $\mathbb{Z}_{p^2}\oplus \mathbb{Z}_{p}$ & $\mathbb{Z}_{p^2}\oplus \mathbb{Z}_{p}^{(2)}$ &$\mathbb{Z}_p^{(4)}\oplus\mathbb Z_{p^2} $ \\
\hline
 $G_{18} $  &$3$  &$\mathbb{Z}_{p}^{(2)}$& $\mathbb{Z}_{p}^{(2)}$ &$\mathbb{Z}_{p}^{(2)}$  &  $\mathbb{Z}_{p^2}\oplus \mathbb{Z}_{p}$ & $\mathbb{Z}_{p^2}\oplus \mathbb{Z}_{p}^{(2)}$ &$\mathbb{Z}_p^{(4)}\oplus\mathbb Z_{p^2} $ \\
\hline
   $G_{19} $ &$3$  &$\mathbb{Z}_{p}$& $\mathbb{Z}_{p^2}$ &$\mathbb{Z}_{p}^{(2)}$  &  $\mathbb{Z}_{p^2}\oplus \mathbb{Z}_{p}$ & $\mathbb{Z}_{p^2}\oplus \mathbb{Z}_{p}^{(2)} $& $\mathbb{Z}_p^{(3)}\oplus\mathbb Z_{p^2} $ \\
\hline
  $G_{20}$  &$3$ &$\mathbb{Z}_{p}^{(2)}$ & $\mathbb{Z}_{p}^{(2)}$ &$\mathbb{Z}_{p}^{(2)}$ & $\mathbb{Z}_{p^2}\oplus \mathbb{Z}_{p}$ & $\mathbb{Z}_{p^2}\oplus \mathbb{Z}_{p}^{(2)}$ &$\mathbb{Z}_p^{(4)}\oplus\mathbb Z_{p^2} $ \\
\hline
 $G_{21}$ &$3$&$\mathbb{Z}_{p}$& $ \mathbb{Z}_{p}^{(2)}$ &$\mathbb{Z}_{p}^{(2)}$ &  $ \mathbb{Z}_{p^2}\oplus \mathbb{Z}_{p}$ & $\mathbb{Z}_{p^2}\oplus \mathbb{Z}_{p}^{(2)} $& $\mathbb{Z}_p^{(3)}\oplus\mathbb Z_{p^2} $
 \\
\hline
$G_{22}$ &$3$&$\mathbb{Z}_{p}$& $ \mathbb{Z}_{p^2}$ &$\mathbb{Z}_{p}^{(2)}$ &  $ \mathbb{Z}_{p^2}\oplus \mathbb{Z}_{p}$& $ \mathbb{Z}_{p^2}\oplus \mathbb{Z}_{p}^{(2)}$&$\mathbb{Z}_p^{(3)}\oplus\mathbb Z_{p^2} $ \\
\hline
 $G_{23}$   &$3$ &$\mathbb{Z}_{p}$&$\mathbb Z_{p^2}$ &$ \mathbb{Z}_{p}^{(2)}$ &  $\mathbb{Z}_{p^2}\oplus \mathbb{Z}_{p}$ & $\mathbb{Z}_{p^2}\oplus \mathbb{Z}_{p}^{(2)}$ &$\mathbb{Z}_p^{(3)}\oplus\mathbb Z_{p^2} $ \\
\hline
$G_{24} $  &$3$  &$1$& $\mathbb {Z}_{p}$ &$ \mathbb{Z}_{p^2 }$  &   $\mathbb{Z}_{p^2}\oplus \mathbb{Z}_{p}$ &  $\mathbb{Z}_{p^2}\oplus \mathbb{Z}_{p}^{(2)}$  &$ \mathbb{Z}_{p}^{(2)} \oplus\mathbb{Z}_{p^2 }$ \\
\hline
 $G_{25} $  &$2$  &$\mathbb{Z}_{p}$& $ \mathbb{Z}_{p^2}\oplus \mathbb{Z}_{p}$ & $\mathbb{Z}_{p}$ & $\mathbb{Z}_{p^3}\oplus \mathbb{Z}_{p}$& $\mathbb{Z}_{p^3}\oplus \mathbb{Z}_{p}^{(2)}$  & $\mathbb{Z}_{p}^{(3)} \oplus\mathbb{Z}_{p^3 }$\\
\hline
 $G_{26} $  &$1$  &$\mathbb{Z}_{p}$& $ \mathbb{Z}_{p^4}\oplus \mathbb{Z}_{p}$ &$1$ &  $\mathbb{Z}_{p^4}\oplus \mathbb{Z}_{p}$ &  $\mathbb{Z}_{p^4}\oplus \mathbb{Z}_{p}^{(2)}$ &$\mathbb{Z}_{p}^{(3)} \oplus\mathbb{Z}_{p^4 }$ \\
\hline
$G_{27} $   &$2$ &$1$&  $\mathbb{Z}_{p^3}$ &$\mathbb{Z}_p$  &$\mathbb{Z}_{p^3}\oplus \mathbb{Z}_{p}$&  $\mathbb{Z}_{p^3}\oplus \mathbb{Z}_{p}^{(2)}$ &$\mathbb{Z}_{p}^{(2)} \oplus\mathbb{Z}_{p^3 }$ \\
\hline
  $G_{28} $  &$4$&$\mathbb{Z}_{p}^{(3)}$&  $\mathbb{Z}_{p}$ &$\mathbb{Z}_p^{(3)}$  & $\mathbb{Z}_p^{(2)}$ & $\mathbb{Z}_{p}^{(3)}$ &$\mathbb{Z}_p^{(6)}$ \\
\hline
 $G_{29_{a}} $  &$4$ &$\mathbb{Z}_{p}$& $\mathbb{Z}_p$ &$\mathbb{Z}_p^{(3)}$  &  $\mathbb{Z}_{p}^{(2)}$ & $\mathbb{Z}_{p}^{(3)}$ &$\mathbb{Z}_p^{(4)}$ \\
\hline
$G_{30} $  &$4$ &$\mathbb{Z}_{p}$& $\mathbb {Z}_{p}$ &$ \mathbb{Z}_p^{(3)}$  &   $\mathbb{Z}_p^{(2)}$ &  $ \mathbb{Z}_{p}^{(3)}$  &$ \mathbb{Z}_{p}^{(4)}$ \\
\hline
 $G_{31} $  &$4$ &$\mathbb{Z}_{p}^{(3)}$& $ \mathbb{Z}_{p}$ & $\mathbb{Z}_p^{(3)}$ &  $\mathbb{Z}_p^{(2)}$ & $\mathbb{Z}_{p}^{(3)}$  & $ \mathbb{Z}_{p}^{(6)}$\\
\hline
\end{tabular}
\end{table}
\begin{table}[]
\centering
\begin{tabular}{|c|c|c|c|c|c|c|c|c|}
 \hline
\textbf{Type of} \large{$G$} &\large{$cl(G)$} &\large{$\mathcal{M}(G)$} &\large{$Z(G)$}&\large{ $G'$}&\large{$G^{ab}$}&\large{$\bigtriangledown(G)$} &\large{$J_{2}(G)$}\\
\hline
\hline
 $G_{32_{a}} $  &$4$  &$\mathbb{Z}_{p}$& $\mathbb{Z}_{p}$ &$\mathbb{Z}_p^{(3)}$ &  $\mathbb{Z}_p^{(2)}$ &  $\mathbb{Z}_{p}^{(3)}$ &$\mathbb{Z}_p^{(4)}$ \\
\hline
 $G_{33_{b}} $   &$4$ &$\mathbb{Z}_{p}$&  $\mathbb{Z}_{p}$ &$\mathbb{Z}_p^{(3)}$  &  $\mathbb{Z}_p^{(2)}$ &  $\mathbb{Z}_{p}^{(3)}$ &$\mathbb{Z}_p^{(4)}$ \\
\hline
  $G_{34} $  &$2$&$\mathbb{Z}_{p}^{(6)}$ &  $\mathbb{Z}_{p}^{(2)}$ &$\mathbb{Z}_p^{(2)}$  & $\mathbb{Z}_p^{(3)}$ & $\mathbb{Z}_{p}^{(6)}$ &$\mathbb{Z}_p^{(12)}$ \\
\hline
 $G_{35} $  &$2$&$\mathbb{Z}_{p}^{(4)}$ & $\mathbb{Z}_{p^2}\oplus \mathbb{Z}_{p}$ &$\mathbb{Z}_p$  &  $\mathbb{Z}_{p}^{(2)}\oplus \mathbb{Z}_{p^2}$ & $\mathbb{Z}_{p^2}\oplus\mathbb{Z}_p^{(5)}$ &$\mathbb{Z}_p^{(9)}\oplus\mathbb{Z}_{p^2}$ \\
\hline
  $G_{36} $  &$2$ &$\mathbb{Z}_{p}^{(3)}$ &  $\mathbb{Z}_{p}^{(2)}$ &$\mathbb{Z}_p^{(2)}$  &$\mathbb{Z}_{p}^{(3)}$ & $\mathbb{Z}_{p}^{(6)}$ &$\mathbb{Z}_{p}^{(9)}$ \\
\hline
 $G_{37} $  &$2$ &$\mathbb{Z}_{p}^{(4)}$ & $\mathbb{Z}_{p}^{(3)}$ &$\mathbb{Z}_p$  &  $\mathbb{Z}_{p}^{(2)}\oplus \mathbb{Z}_{p^2}$ & $\mathbb{Z}_{p^2}\oplus\mathbb{Z}_p^{(5)}$ &$\mathbb{Z}_p^{(9)}\oplus\mathbb{Z}_{p^2}$ \\
\hline
 $G_{38} $  &$2$ &$\mathbb{Z}_{p}^{(3)}$ & $\mathbb{Z}_{p}^{(2)}$ &$\mathbb{Z}_p^{(2)}$  &  $\mathbb{Z}_{p}^{(3)}$ & $\mathbb{Z}_{p}^{(6)}$ &$\mathbb{Z}_p^{(9)}$ \\
\hline
  $G_{39} $  &$2$ &$\mathbb{Z}_{p}^{(3)}$&  $\mathbb{Z}_{p}^{(2)}$ &$\mathbb{Z}_p^{(2)}$  & $\mathbb{Z}_{p}^{(3)}$ & $\mathbb{Z}_{p}^{(6)}$ &$\mathbb{Z}_p^{(9)}$ \\
\hline
 $G_{40} $   &$2$ &$\mathbb{Z}_{p}^{(3)}$& $\mathbb{Z}_{p}^{(3)}$ &$\mathbb{Z}_{p}$  &  $\mathbb{Z}_{p}^{(2)}\oplus \mathbb{Z}_{p^2}$ & $\mathbb{Z}_{p^2}\oplus\mathbb{Z}_p^{(5)}$ &$\mathbb{Z}_p^{(8)}\oplus\mathbb{Z}_{p^2}$ \\
\hline
$G_{41} $   &$2$ &$\mathbb{Z}_{p}^{(2)}$& $\mathbb{Z}_{p}^{(2)}$ &$\mathbb{Z}_p^{(2)}$  &  $\mathbb{Z}_p^{(3)}$ & $\mathbb{Z}_p^{(6)}$ &$\mathbb{Z}_p^{(8)}$ \\
\hline
  $G_{42} $  &$2$ &$\mathbb{Z}_{p}$&  $\mathbb{Z}_{p}^{(2)}$ &$\mathbb{Z}_p^{(2)}$  & $\mathbb{Z}_p^{(3)}$ & $\mathbb{Z}_p^{(6)}$ &$\mathbb{Z}_p^{(7)}$ \\
\hline
 $G_{43} $  &$1$ &$\mathbb{Z}_{p}^{(2)}\oplus\mathbb{Z}_{p^{2}}$ & $\mathbb{Z}_{p^2}^{(2)}\oplus\mathbb{Z}_{p}$ &$1$  &  $\mathbb{Z}_{p^2}^{(2)}\oplus\mathbb{Z}_{p}$ & $\mathbb{Z}_{p^2}^{(3)}\oplus\mathbb{Z}_{p}^{(3)}$ &$\mathbb{Z}_{p^2}^{(4)}\oplus\mathbb{Z}_{p}^{(5)}$ \\
\hline
 $G_{44} $   &$2$ &$\mathbb{Z}_{p}$& $\mathbb{Z}_{p}^{(2)}$ &$\mathbb{Z}_p^{(2)}$  &  $\mathbb{Z}_{p}^
{(3)}$ & $\mathbb{Z}_{p}^{(6)}$ &$\mathbb{Z}_p^{(7)}$ \\
\hline
  $G_{45} $  &$2$ &$\mathbb{Z}_{p}\oplus\mathbb{Z}_{p^{2}}$ &  $\mathbb{Z}_{p^2}\oplus\mathbb{Z}_{p}$ &$\mathbb{Z}_p$  & $\mathbb{Z}_{p^2}\oplus\mathbb{Z}_{p}^{(2)}$ & $\mathbb{Z}_{p^2}\oplus\mathbb{Z}_p^{(5)}$ &$\mathbb{Z}_p^{(6)}\oplus\mathbb{Z}_{p^2}^{(2)}$ \\
\hline
 $G_{46} $   &$2$& $\mathbb{Z}_p$ &$\mathbb{Z}_p^{(2)}$  &  $\mathbb{Z}_{p}^{(2)}$ &  $\mathbb{Z}_{p}^{(3)}$ &$\mathbb{Z}_p^{(6)}$& $ \mathbb{Z}_{p}^{(7)}$ \\
\hline
$G_{47} $   &$2$ & $\mathbb{Z}_{p}^{(2)}$ &$\mathbb{Z}_{p^2}\oplus\mathbb{Z}_{p}$  &  $ \mathbb{Z}_{p}$ &  $\mathbb{Z}_{p^2}\oplus\mathbb{Z}_{p}^{(2)}$ &$\mathbb{Z}_{p^2}\oplus\mathbb{Z}_{p}^{(5)}$& $\mathbb{Z}_{p^2}\oplus\mathbb{Z}_{p}^{(7)}$ \\
\hline
\small{$G_{48,1}=G_{48_{k}}$ if $k\neq\frac{p-1}{2}$}    &$2$ &$\mathbb{Z}_p$& $ \mathbb{Z}_{p}^{(2)}$ &  $\mathbb{Z}_p^{(2)}$ &   $\mathbb{Z}_{p}^{(3)}$ & $ \mathbb{Z}_{p}^{(6)}$& $\mathbb{Z}_{p}^{(7)}$\\
 \hline
 \small{$G_{48,2}=G_{48_{k}}$ if $k=\frac{p-1}{2}$}   &$2$ &$\mathbb{Z}_{p^{2}}$& $\mathbb{Z}_p^{(2)}$ & $\mathbb{Z}_p^{(2)}$  &  $\mathbb{Z}_{p}^{(3)}$ & $\mathbb{Z}_{p}^{(6)}$& $\mathbb{Z}_{p^2}\oplus\mathbb{Z}_{p}^{(6)}$\\
 \hline
$G_{49}$ &$2$&$\mathbb{Z}_{p^{2}}$& $\mathbb{Z}_p^{(2)}$ & $\mathbb{Z}_p^{(2)}$ &  $\mathbb{Z}_{p}^{(3)}$ & $\mathbb{Z}_{p}^{(6)}$  &$  \mathbb Z_{p}^{(6)}\oplus\mathbb Z_{p^{2}} $\\
\hline
 $G_{50_{k}}$   &$2$&$\mathbb{Z}_{p}$&$ \mathbb{Z}_p^{(2)}$ & $\mathbb{Z}_p^{(2)}$ &   $\mathbb{Z}_{p}^{(3)}$ & $\mathbb{Z}_{p}^{(6)}$&$\mathbb Z_{p}^{(7)}$\\
\hline
$G_{51}$  &$1$&$\mathbb{Z}_p^{(3)}$&$ \mathbb{Z}_{p^3}\oplus\mathbb{Z}_p^{(2)} $ &$ 1 $ &  $\mathbb{Z}_{p^3}\oplus\mathbb{Z}_p^{(2)}$ & $\mathbb{Z}_{p^3}\oplus\mathbb{Z}_p^{(5)}$&$\mathbb Z_{p}^{(8)}\oplus\mathbb Z_{p^3}$\\
\hline
 $G_{52}$   &$2$ &$\mathbb{Z}_p^{(2)}$&$\mathbb{Z}_{p^3}$ &$\mathbb Z_{p}$ & $\mathbb{Z}_{p^2}\oplus\mathbb{Z}_p^{(2)}$ & $\mathbb{Z}_{p^2}\oplus\mathbb{Z}_p^{(5)}$ &$\mathbb{Z}_{p^2}\oplus\mathbb{Z}_p^{(7)}$\\
 \hline
 $G_{53}$   &$2$&$\mathbb{Z}_p^{(2)}$& $\mathbb{Z}_{p^2}\oplus\mathbb{Z}_p$  &$\mathbb{Z}_{p}$ &  $\mathbb{Z}_{p^2}\oplus\mathbb{Z}_p^{(2)}$  & $\mathbb{Z}_{p^2}\oplus\mathbb{Z}_p^{(5)}$ &$\mathbb{Z}_{p^2}\oplus\mathbb{Z}_p^{(7)}$ \\
\hline
 $G_{54}$  &$3$ &$\mathbb{Z}_p^{(4)}$& $ \mathbb{Z}_{p}^{(2)}$ &$ \mathbb{Z}_{p}^{(2)}$ &  $ \mathbb{Z}_p^{(3)}$ & $ \mathbb{Z}_{p}^{(6)}$ &$\mathbb{Z}_{p}^{(10)}$ \\
\hline
$G_{55}$  &$3$ &$\mathbb{Z}_p^{(3)}$ & $ \mathbb{Z}_{p^2}$ &$  \mathbb{Z}_{p}^{(2)}$ &  $\mathbb{Z}_p^{(3)}$& $\mathbb{Z}_{p}^{(6)}$&$\mathbb{Z}_{p}^{(9)}$ \\
\hline
 $G_{56}$   &$3$ &$\mathbb{Z}_p^{(3)}$&$\mathbb{Z}_{p}^{(2)}$ &$  \mathbb{Z}_{p}^{(2)}$ &  $ \mathbb{Z}_p^{(3)}$ & $ \mathbb{Z}_{p}^{(6)}$ &$ \mathbb{Z}_{p}^{(9)}$ \\
\hline
 $G_{57} $  &$3$ &$\mathbb{Z}_p^{(3)}$ &$ \mathbb{Z}_{p}^{(2)}$& $ \mathbb{Z}_{p}^{(2)}$ &  $ \mathbb{Z}_p^{(3)}$ & $\mathbb{Z}_{p}^{(6)}$  & $\mathbb{Z}_{p}^{(9)}$\\
\hline
 $G_{58} $  &$3$ &$\mathbb{Z}_p^{(3)}$ & $\mathbb {Z}_{p}^{(2)} $ &$ \mathbb{Z}_{p}^{(2)}$ & $\mathbb{Z}_p^{(3)}$ &  $\mathbb{Z}_{p}^{(6)}$ & $ \mathbb{Z}_{p}^{(9)}$ \\
  \hline
 $G_{59} $ &$3$  &$\mathbb{Z}_p^{(4)}$ &$\mathbb{Z}_{p}$ &$ \mathbb{Z}_{p}^{(2)}$ &  $\mathbb{Z}_p^{(3)}$ & $\mathbb{Z}_{p}^{(6)}$& $\mathbb{Z}_{p}^{(10)}$ \\
\hline
$G_{60} $   &$3$ &$\mathbb{Z}_p^{(3)}$&$\mathbb{Z}_{p}$ &$ \mathbb{Z}_{p}^{(2)}$& $\mathbb{Z}_p^{(3)}$ &  $\mathbb{Z}_{p}^{(6)}$& $\mathbb{Z}_{p}^{(9)}$ \\
\hline
 $G_{61} $  &$3$ &$\mathbb{Z}_p^{(3)}$ & $\mathbb{Z}_{p} $ &$ \mathbb{Z}_{p}^{(2)}$ &  $\mathbb{Z}_{p}^{(3)}$ & $\mathbb{Z}_p^{(6)}$  &$\mathbb{Z}_{p}^{(9)}$\\
\hline

 $G_{62} $  &$3$ &$\mathbb{Z}_p^{(3)}$ &  $\mathbb{Z}_{p}$ &$ \mathbb{Z}_{p}^{(2)}$  &  $\mathbb{Z}_p^{(3)}$&$\mathbb{Z}_{p}^{(6)}$&$\mathbb{Z}_{p}^{(9)}$ \\
\hline
 $G_{63} $   &$3$&$\mathbb{Z}_p^{(3)}$&  $\mathbb{Z}_{p}$ &$ \mathbb{Z}_{p}^{(2)}$& $\mathbb{Z}_{p}^{(3)}$ & $\mathbb{Z}_{p}^{(6)}$ &$\mathbb{Z}_{p}^{(9)}$ \\
\hline
 $G_{64} $   &$2$ &$\mathbb{Z}_p^{(7)}$&  $\mathbb{Z}_{p}^{(3)}$ &$ \mathbb{Z}_{p}$  & $\mathbb{Z}_p^{(4)}$ & $\mathbb{Z}_{p}^{(10)}$ &$\mathbb{Z}_{p}^{(17)}$ \\
\hline
 $G_{65} $   &$3$ &$\mathbb{Z}_p^{(5)}$& $\mathbb{Z}_{p}$ &$\mathbb{Z}_{p}$  &  $\mathbb{Z}_p^{(4)}$ & $\mathbb{Z}_{p}^{(10)}$ &$\mathbb{Z}_{p}^{(15)}$ \\
\hline
  $G_{66}$  &$1$&$\mathbb{Z}_p^{(6)}$ & $\mathbb{Z}_p^{(3)}\oplus\mathbb{Z}_{p^2}$  &$1$ &  $\mathbb{Z}_{p^2}\oplus \mathbb{Z}_p^{(3)}$& $\mathbb{Z}_{p^2}\oplus \mathbb{Z}_p^{(9)}$ &$\mathbb{Z}_p^{(15)}\oplus\mathbb{Z}_{p^2}$ \\
\hline
 $G_{67}$  &$2$ &$\mathbb{Z}_p^{(5)}$& $ \mathbb{Z}_{p}^{(3)}$ &$\mathbb{Z}_p$ &  $ \mathbb{Z}_p^{(4)}$ & $ \mathbb{Z}_{p}^{(10)}$ &$\mathbb{Z}_p^{(15)}$ \\
\hline
$G_{68}$  &$2$ &$\mathbb{Z}_{p}^{(5)}$& $ \mathbb{Z}_{p}^{(2)}\oplus \mathbb{Z}_{p}$ &$\mathbb{Z}_p$ &  $ \mathbb{Z}_p^{(4)}$& $\mathbb{Z}_{p}^{(10)}$&$\mathbb{Z}_p^{(15)}$ \\
\hline
 $G_{69}$   &$3$ &$\mathbb{Z}_{p}^{(5)}$&$\mathbb{Z}_{p}$ &$ \mathbb{Z}_p$ &  $ \mathbb{Z}_p^{(4)}$ & $\mathbb{Z}_{p}^{(10)}$ &$\mathbb{Z}_{p}^{(15)}$ \\
\hline
$G_{70} $   &$1$ &$\mathbb{Z}_{p}^{(10)}$& $\mathbb {Z}_{p}^{(5)} $ &$1$  &   $\mathbb{Z}_p^{(5)}$ &  $\mathbb{Z}_{p}^{(15)}$  &$\mathbb{Z}_{p}^{(25)}$ \\
\hline
\end{tabular}
\end{table}
\begin{table}[]
\centering
\caption{Fig. 2.}\label{ta2}
\begin{tabular}{|c|c|c|c|c|c|c|c|}
 \hline
\textbf{Type of} \large{$G$} &\large{$cl(G)$} &\large{$\mathcal{M}(G)$} &\large{$G\wedge G$}&\large{ $G\otimes G$}&\large{$Z^{\wedge}(G)$}&\large{${Z}^{\otimes}(G)$} \\
\hline
\hline
$G_{1}$  &$1$ &$1$& $1$ & $ \mathbb{Z}_{p^5 }$  &  $ \mathbb{Z}_{p^5 }$ & $ 1$\\
 \hline
$G_{2}$ &$2$ & $\mathbb{Z}_{p}^{(3)} $& $\mathbb Z_{p^{2}}\oplus\mathbb{Z}_{p}^{(2)}$ &$ \mathbb Z_{p^{2}}^{(4)} \oplus \mathbb Z_{p}^{(2)}$ &  $1$ & $1$  \\
\hline
 $G_{3}$  &$3$  & $\mathbb{Z}_{p}^{(3)}$&$ \mathbb{Z}_{p}^{(6)}$ &$\mathbb{Z}_{p}^{(9)}$ &  $1$ & $1$\\
\hline
$G_{4}$  &$3$  & $ \mathbb{Z}_{p}$&$ \mathbb{Z}_p^{(4)}$ &$ \mathbb{Z}_{p}^{(7)} $ &  $\mathbb{Z}_{p}$ & $\mathbb{Z}_p$\\
\hline
 $G_{5}$ &$3$    & $\mathbb{Z}_{p}$&$\mathbb{Z}_p^{(4)}$ &$\mathbb Z_{p}^{(7)}$ &  $\mathbb{Z}_p$ & $\mathbb{Z}_p$ \\
 \hline
 $G_{6}$  &$3$ &$\mathbb{Z}_{p}$& $\mathbb{Z}_{p}^{(4)}$  &$\mathbb{Z}_{p}^{(7)}$ &  $\mathbb{Z}_p$ & $\mathbb{Z}_p$  \\
\hline
 $G_{7}$  &$3$&$1$& $ \mathbb{Z}_{p}^{(3)}$ &$\mathbb{Z}_{p}^{(6)}$ &  $ \mathbb{Z}_{p}^{(2)}$ & $ \mathbb{Z}_{p}^{(2)}$  \\
\hline
$G_{8}$ &$3$  &$1$& $\mathbb{Z}_{p}^{(3)}$ &$ \mathbb{Z}_{p}^{(6)}$ &  $\mathbb{Z}_{p}^{(2)}$& $\mathbb{Z}_{p}^{(2)}$ \\
\hline
 $G_{9}$   &$3$ &$1$&$\mathbb{Z}_{p}^{(3)}$ &$\mathbb{Z}_{p}^{(6)}$ &  $\mathbb{Z}_{p}^{(2)}$ & $\mathbb{Z}_{p}^{(2)}$  \\
\hline
 $G_{10} $  &$3$ &$\mathbb{Z}_{p}$ & $ \mathbb{Z}_{p^2 }\oplus \mathbb{Z}_{p}^{(2)}$ & $\mathbb{Z}_{p^2 }\oplus \mathbb{Z}_{p}^{(5)}$ &  $1$ & $1$  \\
\hline
 \small{$G_{11,1}=G_{11_{k}} $if $k\neq \frac{p-1}{2}$}  &$3$ & $1$ & $ \mathbb{Z}_{p}^{(3)}$ &$ \mathbb{Z}_{p}^{(6)}$ &  $\mathbb Z_{p}^{(2)}$ &  $\mathbb Z_{p}^{(2)}$ \\
\hline
 \small{$G_{11,2}=G_{11_{k}} $if $k=\frac{p-1}{2}$} &$3$ &$\mathbb Z_{p}$ & $\mathbb{Z}_{p^2 }\oplus \mathbb{Z}_{p}^{(2)} $ &$\mathbb{Z}_{p^2 }\oplus \mathbb{Z}_{p}^{(5)}$ &  $1$ &  $1$ \\
\hline
 $G_{12_{k}} $  &$3$&$1$&$\mathbb{Z}_{p}^{(3)}$ &$\mathbb{Z}_{p}^{(6)}$ &  $\mathbb Z_{p}^{(2)}$ & $\mathbb Z_{p}^{(2)}$ \\
\hline
 $G_{13} $  &$1$ & $\mathbb{Z}_{p^2}$&$\mathbb{Z}_{p^2}$ &$\mathbb Z_{p^3}\oplus\mathbb Z_{p^2}^{(3)}$&$\mathbb Z_{p}$ &$1$  \\
\hline
 $G_{14} $ &$2$&$\mathbb{Z}_{p}$ & $\mathbb Z_{p^2} $ &$\mathbb Z_{p^2}^{(4)}$ &  $\mathbb{Z}_{p}$ & $\mathbb{Z}_{p}$  \\
\hline
 $G_{15} $ &$2$ &$\mathbb{Z}_{p}^{(2)}$ &  $\mathbb Z_{p}^{(3)}$ &$\mathbb{Z}_{p^3}\oplus \mathbb Z_{p}^{(5)} $  &  $\mathbb{Z}_{p^{2}}$&$1$ \\
\hline
 $G_{16} $   &$3$ &$\mathbb{Z}_{p}^{(2)}$&  $\mathbb Z_{p}^{(4)}$ &$\mathbb{Z}_{p^2}\oplus\mathbb{Z}_{p}^{(6)}$&  $\mathbb{Z}_{p}$& $1$  \\
\hline
 $G_{17} $   &$3$ &$\mathbb{Z}_{p}^{(2)}$&  $\mathbb{Z}_{p^2}\oplus\mathbb{Z}_{p}^{(2)}$ &$\mathbb{Z}_{p^2}^{(2)}\oplus\mathbb{Z}_{p}^{(4)}$  & $1$  & $1$\\
\hline
 $G_{18} $  &$3$  &$\mathbb{Z}_{p}^{(2)}$& $\mathbb{Z}_{p^2}\oplus\mathbb{Z}_{p}^{(2)}$ &$\mathbb{Z}_{p^2}^{(2)}\oplus\mathbb{Z}_{p}^{(4)}$  &  $1$ & $1$  \\
\hline
   $G_{19} $ &$3$  &$\mathbb{Z}_{p}$& $\mathbb{Z}_{p}^{(3)}$ &$\mathbb{Z}_{p^2}\oplus\mathbb{Z}_{p}^{(5)}$  &  $\mathbb{Z}_{p^2}$ & $\mathbb{Z}_{p}$ \\
\hline
  $G_{20}$  &$3$ &$\mathbb{Z}_{p}^{(2)}$ & $\mathbb{Z}_p^{(4)}$  &$\mathbb{Z}_{p^2}\oplus\mathbb{Z}_{p}^{(6)}$ & $\mathbb{Z}_{p}$ & $1$  \\
\hline
 $G_{21}$ &$3$&$\mathbb{Z}_{p}$& $ \mathbb{Z}_{p}^{(3)}$ &$\mathbb{Z}_{p^2}\oplus\mathbb{Z}_{p}^{(5)}$ &  $ \mathbb{Z}_{p}^{(2)}$ & $ \mathbb{Z}_{p}$  \\
\hline
$G_{22}$ &$3$&$\mathbb{Z}_{p}$& $ \mathbb{Z}_{p}^{(3)}$ &$\mathbb{Z}_{p^2}\oplus\mathbb{Z}_{p}^{(5)}$ &  $ \mathbb{Z}_{p^2}$& $ \mathbb{Z}_{p}$ \\
 \hline
 $G_{23}$   &$3$ &$\mathbb{Z}_{p}$&$\mathbb{Z}_{p}^{(3)}$ &$ \mathbb{Z}_{p^2}\oplus\mathbb{Z}_{p}^{(5)}$ &  $ \mathbb{Z}_{p^2}$ & $ \mathbb{Z}_{p}$  \\
\hline
$G_{24} $  &$3$  &$1$& $\mathbb {Z}_{p^2} $ &$ \mathbb{Z}_{p^2}^{(2)}\oplus\mathbb{Z}_{p}^{(2)}$  &   $ \mathbb{Z}_{p}$ &  $ \mathbb{Z}_{p}$   \\
\hline
 $G_{25} $  &$2$  &$\mathbb{Z}_{p}$& $ \mathbb{Z}_{p^2}$ & $\mathbb{Z}_{p^2}\oplus\mathbb{Z}_{p^3}\oplus\mathbb{Z}_{p}^{(2)}$ & $\mathbb{Z}_{p}$& $1$  \\
\hline
 $G_{26} $  &$1$  &$\mathbb{Z}_{p}$& $\mathbb{Z}_{p}$ &$\mathbb{Z}_{p^4}\oplus\mathbb{Z}_p^{(3)}$ &  $\mathbb{Z}_{p^3}$ &$1$ \\
 \hline
 $G_{27} $  &$2$  &$1$& $\mathbb{Z}_{p}$ &$\mathbb{Z}_{p^3}\oplus\mathbb{Z}_p^{(3)}$ &  $\mathbb{Z}_{p^3}$ &$\mathbb{Z}_{p}$ \\
  \hline
  $G_{28} $  &$4$&$\mathbb{Z}_{p}^{(3)}$&  $E_{1}\times\mathbb{Z}_{p}^{(3)}$ &$E_{1}\times\mathbb{Z}_{p}^{(6)}$  & $1$ & $1$  \\
\hline
 $G_{29_{a}} $  &$4$ &$\mathbb{Z}_{p}$& $\mathbb{Z}_{p}^{(4)}$ &$\mathbb{Z}_{p}^{(7)}$  &  $\mathbb{Z}_{p}$ & $\mathbb{Z}_{p}$ \\
\hline
$G_{30} $  &$4$ &$\mathbb{Z}_{p}$& $\mathbb{Z}_{p}^{(4)} $ &$ \mathbb{Z}_{p}^{(7)}$  &   $ \mathbb{Z}_{p}$ &  $  \mathbb{Z}_{p}$   \\
\hline
 $G_{31} $  &$4$ &$\mathbb{Z}_{p}^{(3)}$& $E_{1}\times\mathbb{Z}_{p}^{(3)}$ & $E_{1}\times\mathbb{Z}_{p}^{(6)}$ &  $1$ & $1$  \\
\hline
 $G_{32_{a}} $  &$4$  &$\mathbb{Z}_{p}$& $\mathbb{Z}_{p}^{(4)}$ &$\mathbb{Z}_p^{(7)}$ &  $\mathbb{Z}_{p}$ &  $\mathbb{Z}_{p}$  \\
\hline
 $G_{33_{b}} $   &$4$ &$\mathbb{Z}_{p}$&  $\mathbb{Z}_{p}^{(4)}$ &$\mathbb{Z}_p^{(7)}$  &  $\mathbb{Z}_{p}$ &  $\mathbb{Z}_{p}$ \\
\hline
  $G_{34} $  &$2$&$\mathbb{Z}_{p}^{(6)}$ &  $\mathbb{Z}_{p}^{(8)}$ &$\mathbb{Z}_p^{(14)}$  & $1$ & $1$  \\
\hline
$G_{35} $  &$2$&$\mathbb{Z}_{p}^{(4)}$ & $\mathbb{Z}_{p}^{(5)}$ &$\mathbb{Z}_{p^2}\oplus\mathbb{Z}_{p}^{(10)}$  &$\mathbb{Z}_{p}$  &$1$ \\
\hline
$G_{36} $  &$2$ &$\mathbb{Z}_{p}^{(3)}$ &  $\mathbb{Z}_{p}^{(5)}$ &$\mathbb{Z}_{p}^{(11)}$  &$\mathbb{Z}_p$ & $\mathbb{Z}_{p}$ \\
 \hline
 $G_{37} $  &$2$ &$\mathbb{Z}_{p}^{(4)}$ & $\mathbb{Z}_{p}^{(5)}$ &$\mathbb{Z}_{p^2}\oplus\mathbb{Z}_{p}^{(10)}$  &$\mathbb{Z}_{p}$ & $1$  \\
\hline
 $G_{38} $  &$2$ &$\mathbb{Z}_{p}^{(3)}$ & $\mathbb{Z}_{p}^{(5)}$ &$\mathbb{Z}_{p}^{(11)}$  &  $\mathbb{Z}_{p}$ & $\mathbb{Z}_{p}$  \\
\hline
  $G_{39} $  &$2$ &$\mathbb{Z}_{p}^{(3)}$&  $\mathbb{Z}_{p}^{(5)}$ &$\mathbb{Z}_p^{(11)}$  & $\mathbb{Z}_{p}$& $\mathbb{Z}_{p}$  \\
\hline
 $G_{40} $   &$2$ &$\mathbb{Z}_{p}^{(3)}$& $\mathbb{Z}_{p^2}\oplus\mathbb{Z}_p^{(2)} $ &$\mathbb{Z}_{p^2}^{(2)}\oplus\mathbb{Z}_p^{(7)}$  &  $1$ & $1$  \\
\hline
$G_{41} $   &$2$ &$\mathbb{Z}_{p}^{(2)}$& $\mathbb{Z}_{p^2}\oplus\mathbb{Z}_p^{(2)}$ &$\mathbb{Z}_{p^2}\oplus\mathbb{Z}_p^{(8)}$  &  $1$ & $1$\\
\hline
  $G_{42} $  &$2$ &$\mathbb{Z}_{p}$&  $\mathbb{Z}_{p}^{(3)}$ &$\mathbb{Z}_p^{(9)}$  & $\mathbb{Z}_{p}^{(2)}$ & $\mathbb{Z}_{p}^{(2)}$  \\
\hline
 $G_{43} $  &$1$ &$\mathbb{Z}_{p}^{(2)}\oplus\mathbb{Z}_{p^{2}}$ & $\mathbb{Z}_{p^2}\oplus\mathbb{Z}_p^{(2)}$ &$\mathbb{Z}_{p^2}^{(4)}\oplus\mathbb{Z}_p^{(5)}$  &  $1$ & $1$  \\
\hline
 $G_{44} $   &$2$ &$\mathbb{Z}_{p}$& $\mathbb{Z}_p^{(3)}$ &$\mathbb{Z}_p^{(9)}$  &  $\mathbb{Z}_{p}^
{(2)}$ & $\mathbb{Z}_{p}^
{(2)}$ \\
\hline
\end{tabular}
\end{table}
\begin{table}[h!]
\centering
\begin{tabular}{|c|c|c|c|c|c|c|c|}
 \hline
\textbf{Type of} \large{$G$} &\large{$cl(G)$} &\large{$\mathcal{M}(G)$} &\large{$G\wedge G$}&\large{ $G\otimes G$}&\large{$Z^{\wedge}(G)$}&\large{${Z}^{\otimes}(G)$} \\
\hline
\hline
  $G_{45} $  &$2$ &$\mathbb{Z}_{p}\oplus\mathbb{Z}_{p^{2}}$ &  $\mathbb{Z}_{p^2}\oplus\mathbb{Z}_p^{(2)}$ &$\mathbb{Z}_{p^2}^{(2)}\oplus\mathbb{Z}_p^{(7)}$  & $1$ & $1$  \\
\hline
 $G_{46} $   &$2$& $\mathbb{Z}_p$ &$\mathbb{Z}_p^{(3)}$  &  $\mathbb{Z}_p^{(9)}$ &  $\mathbb{Z}_{p}^{(2)}$ &$\mathbb{Z}_{p}^{(2)}$ \\
\hline
$G_{47} $   &$2$ & $\mathbb{Z}_p^{(2)}$ &$\mathbb{Z}_p^{(3)}$  &  $\mathbb{Z}_{p^2}\oplus\mathbb{Z}_p^{(8)}$ &  $\mathbb{Z}_{p}^{(2)}$ &$\mathbb{Z}_{p}$ \\
\hline
\small{$G_{48,1}=G_{48_{k}}$ if $k\neq\frac{p-1}{2}$}    &$2$ &$\mathbb{Z}_p$& $\mathbb{Z}_{p}^{(3)}$ & $\mathbb{Z}_{p}^{(9)}$  &   $\mathbb{Z}_{p}^{(2)}$ & $\mathbb{Z}_{p}^{(2)}$\\
 \hline
  \small{ $G_{48,2}=G_{48_{k}}$ if $k=\frac{p-1}{2}$}   &$2$ &$\mathbb{Z}_{p^{2}}$& $ \mathbb{Z}_{p^2 }\oplus\mathbb{Z}_p^{(2)}$ & $\mathbb{Z}_{p^2 }\oplus\mathbb{Z}_p^{(8)}$  &  $ 1$ & $ 1$\\
 \hline
$G_{49}$ &$2$&$\mathbb{Z}_{p^{2}}$& $\mathbb{Z}_{p^2 }\oplus\mathbb{Z}_p^{(2)}$ &$ \mathbb{Z}_{p^2 }\oplus\mathbb{Z}_p^{(8)}$ &  $1$ & $1$  \\
\hline
 $G_{50_{k}}$   &$2$&$\mathbb{Z}_{p}$&$\mathbb{Z}_p^{(3)}$ &$\mathbb Z_{p}^{(9)}$ &   $\mathbb{Z}_{p}^{(2)}$ & $\mathbb{Z}_{p}^{(2)}$\\
\hline
$G_{51}$  &$1$&$\mathbb{Z}_p^{(3)}$&$ \mathbb{Z}_p^{(3)}$ &$ \mathbb{Z}_{p^3 }\oplus\mathbb Z_{p}^{(8)} $ &  $\mathbb{Z}_{p^2}$ & $1$\\
\hline
 $G_{52}$   &$2$ &$\mathbb{Z}_p^{(2)}$&$\mathbb{Z}_p^{(3)}$ &$\mathbb{Z}_{p^2 }\oplus\mathbb Z_{p}^{(8)}$ & $\mathbb{Z}_{p^2}$ & $\mathbb{Z}_{p}$ \\
 \hline
 $G_{53}$   &$2$&$\mathbb{Z}_p^{(2)}$& $\mathbb{Z}_{p}^{(3)}$  &$\mathbb{Z}_{p^2 }\oplus\mathbb Z_{p}^{(8)}$ & $\mathbb{Z}_{p^2}$   & $\mathbb{Z}_{p}$ \\
\hline
 $G_{54}$  &$3$ &$\mathbb{Z}_p^{(4)}$& $ \mathbb{Z}_{p}^{(6)}$ &$\mathbb{Z}_{p}^{(12)}$ &  $ 1$ & $ 1$  \\
\hline
$G_{55}$  &$3$ &$\mathbb{Z}_p^{(3)}$ & $ \mathbb{Z}_{p}^{(5)}$ &$ \mathbb{Z}_{p}^{(11)}$ &  $\mathbb{Z}_{p}$& $\mathbb{Z}_{p}$ \\
\hline
 $G_{56}$   &$3$ &$\mathbb{Z}_p^{(3)}$&$\mathbb{Z}_{p}^{(5)}$ &$ \mathbb{Z}_{p}^{(11)}$&  $ \mathbb{Z}_{p}$ & $ \mathbb{Z}_{p}$  \\
\hline
 $G_{57} $  &$3$ &$\mathbb{Z}_p^{(3)}$ & $\mathbb{Z}_{p}^{(5)}$ & $\mathbb{Z}_{p}^{(11)}$ &  $ \mathbb{Z}_{p}$ & $\mathbb{Z}_{p}$  \\
\hline
 $G_{58} $  &$3$ &$\mathbb{Z}_p^{(3)}$ & $\mathbb{Z}_{p}^{(5)} $ &$\mathbb{Z}_{p}^{(11)}$ & $\mathbb{Z}_{p}$ &  $\mathbb{Z}_{p}$  \\
\hline
 $G_{59} $ &$3$  &$\mathbb{Z}_p^{(4)}$ &$\mathbb{Z}_{p}^{(6)}$ &$\mathbb{Z}_{p}^{(12)}$ &  $1$ & $1$ \\
\hline
 $G_{60} $   &$3$ &$\mathbb{Z}_p^{(3)}$&$\mathbb{Z}_{p}^{(5)}$ &$\mathbb{Z}_{p}^{(11)}$& $\mathbb{Z}_{p}$ &  $\mathbb{Z}_{p}$\\
\hline
 $G_{61} $  &$3$ &$\mathbb{Z}_p^{(3)}$ & $\mathbb{Z}_{p}^{(5)} $ &$\mathbb{Z}_{p}^{(11)}$ &  $\mathbb{Z}_{p}$ & $\mathbb{Z}_{p}$  \\
\hline
 $G_{62} $  &$3$ &$\mathbb{Z}_p^{(3)}$ &  $\mathbb{Z}_{p}^{(5)}$ &$\mathbb{Z}_{p}^{11)}$  &  $\mathbb{Z}_{p}$&$\mathbb{Z}_{p}$  \\
\hline
 $G_{63} $   &$3$&$\mathbb{Z}_p^{(3)}$&  $\mathbb{Z}_{p}^{(5)}$ &$\mathbb{Z}_{p}^{(11)}$& $\mathbb{Z}_{p}$ & $\mathbb{Z}_{p}$  \\
\hline
 $G_{64} $   &$2$ &$\mathbb{Z}_p^{(7)}$&  $\mathbb{Z}_{p}^{(8)}$ &$\mathbb{Z}_{p}^{(18)}$  & $1$ & $1$ \\
\hline
 $G_{65} $   &$3$ &$\mathbb{Z}_p^{(5)}$& $\mathbb{Z}_{p}^{(6)}$ &$\mathbb{Z}_{p}^{(16)}$  &  $\mathbb{Z}_{p}$ & $\mathbb{Z}_{p}$  \\
 \hline
  $G_{66}$  &$1$&$\mathbb{Z}_p^{(6)}$ & $\mathbb{Z}_{p}^{(6)}$  &$\mathbb{Z}_{p^2}\oplus\mathbb{Z}_{p}^{(15)}$ &  $\mathbb{Z}_{p}$& $1$  \\
\hline
 $G_{67}$  &$2$ &$\mathbb{Z}_p^{(5)}$& $\mathbb{Z}_{p}^{(6)}$ &$\mathbb{Z}_{p}^{(16)}$ &  $ \mathbb{Z}_{p}$ & $ \mathbb{Z}_{p}$  \\
\hline
$G_{68}$  &$2$ &$\mathbb{Z}_{p}^{(5)}$& $\mathbb{Z}_{p}^{(6)}$ &$\mathbb{Z}_p^{(16)}$ &  $ \mathbb{Z}_{p}$& $ \mathbb{Z}_{p}$ \\
\hline
 $G_{69}$   &$3$ &$\mathbb{Z}_{p}^{(5)}$&$\mathbb{Z}_{p}^{(6)}$ &$ \mathbb{Z}_{p}^{(16)}$ &  $ \mathbb{Z}_{p}$ &  $ \mathbb{Z}_{p}$  \\
\hline
$G_{70} $   &$1$ &$\mathbb{Z}_{p}^{(10)}$& $\mathbb{Z}_{p}^{(10)} $ &$ \mathbb{Z}_{p}^{(25)}$  &   $1$ &  $ 1$\\
\hline
\end{tabular}
\end{table}


\begin{thebibliography}{99}%
{\small}


\bibitem{Blyth} R. D. Blyth, F. Fumagalli and M. Morigi, Some structural results on the non-abelian tensor square of groups, J. of Group Theory 13, no. 1 (2010) 83-94.

\bibitem{Brown1}
 R. Brown and J. L. Loday, Excision homotopique en basse dimension, CR Acad. Sci. Paris SI Math 298, no. 15 (1984) 353-356.

\bibitem{B2}
R. Brown and J. L. Loday, Van Kampen theorems for diagrams of spaces, Topology 26 (1987) 311-335.

\bibitem{Brown}
R. Brown, D. L. Johnson and  E. F. Robertson, Some computations of non-abelian tensor products of groups, J. Algebra 111, no. 1 (1987): 177-202.


\bibitem{Ellis5}
G. Ellis, Tensor Products and q-Crossed Modules, J. Lond. Math. Soc. 51, no. 2 (1995) 243-258.
\bibitem{Ghorbanzadeh}
T. J. Ghorbanzadeh, M. Prvizi and P. Niroomand, The non-abelian tensor square of $p$-groups of order $p^4$. Asian Europian J. Math (to appear).
https://doi.org/10.1142/1793557118500845.

\bibitem{Girant}
B. Girant, Die Klassifikation der Gruppen bis zur Ordnung $p^5$, arXiv:1806.07462.

\bibitem{Hatui}
S. Hatui, V. Kakkar and M. K. Yadav, The Schur Multipliers of $p$-Groups of Order $p^5$, arXiv:1804.11308.

\bibitem{Horn}
M. Horn and S. Zandi, Computing Schur multiplier via Lazard correspondence. In preparation.

 \bibitem{James}
 R. James, The groups of order $p^6$ (p an odd prime), Math. Comp. 34 (1980), 613-637.

 \bibitem{Moghaddam}
 M. R. R. Moghaddam and P. Niroomand, Some properties of certain subgroups of tensor squares of p-groups, Comm. Algebra, 40, no. 3 (2012) 1188-1193.

\bibitem{Moravec}
 P. Moravec, The exponents of non-abelian tensor products of groups, J. Pure Appl. Algebra 212, no. 7 (2008) 1840-1848.

\bibitem{Whitehead}
 J. H. Whitehead, A certain exact sequence, Ann. of Math. 52 (1950) 51-110.

\end{thebibliography}
 \end{document}